\documentclass[12pt]{amsart}
\usepackage{amsmath,amsfonts,amsthm,amssymb}
\usepackage{color}
\usepackage{hyperref}
\usepackage{tikz}

\theoremstyle{plain}
\newtheorem{thm}{Theorem}[section]
\newtheorem{lem}[thm]{Lemma}
\newtheorem{prop}[thm]{Proposition}
\newtheorem{cor}[thm]{Corollary}
\theoremstyle{definition}
\newtheorem{defn}[thm]{Definition}
\newtheorem{rmk}[thm]{Remark}

\newcounter{sectiontmp1}
\newcounter{sectiontmp2}
\newcounter{thmtmp1}
\newcounter{thmtmp2}

\renewcommand{\phi}{\varphi}
\newcommand{\comment}[1]{}
\newcommand{\conv}{\!\!\downarrow}
\newcommand{\dv}{\!\!\uparrow}

\newcommand{\concat}{
  \mathord{
    \mathchoice
    {\raisebox{1ex}{\scalebox{.7}{$\frown$}}}
    {\raisebox{1ex}{\scalebox{.7}{$\frown$}}}
    {\raisebox{.7ex}{\scalebox{.5}{$\frown$}}}
    {\raisebox{.7ex}{\scalebox{.5}{$\frown$}}}
  }
}

\newcommand{\N}{\mathbb{N}}
\newcommand{\RCA}{\mathsf{RCA}}
\newcommand{\ATR}{\mathsf{ATR}}
\newcommand{\CA}{\mathsf{CA}}
\newcommand{\AC}{\mathsf{AC}}
\newcommand{\UC}{\mathsf{UC}}
\newcommand{\C}{\mathsf{C}}
\newcommand{\CWO}{\mathsf{CWO}}
\newcommand{\WCWO}{\mathsf{WCWO}}

\renewcommand{\O}{\mathcal{O}}
\newcommand{\KB}{\mathrm{KB}}

\newcommand{\otp}{\mathrm{otp}}
\newcommand{\rk}{\mathrm{rk}}
\renewcommand{\L}{\mathcal{L}}
\newcommand{\M}{\mathcal{M}}
\newcommand{\KDT}{\mathsf{KDT}}
\newcommand{\LPO}{\mathsf{LPO}}

\newcommand{\WQO}{\mathsf{WQO}}
\newcommand{\NDS}{\mathsf{NDS}}
\newcommand{\NIAC}{\mathsf{NIAC}}
\newcommand{\LO}{\mathrm{LO}}
\newcommand{\WO}{\mathrm{WO}}
\newcommand{\range}{\mathrm{range}}
\newcommand{\dom}{\mathrm{dom}}
\newcommand{\id}{\mathrm{id}}
\newcommand{\arith}{\mathrm{arith}}

\title[Some reductions between principles around $\ATR_0$]{Some computability-theoretic reductions between principles around $\ATR_0$}
\date{\today}
\author{Jun Le Goh}
\address{Department of Mathematics, Cornell University, 310 Malott Hall, Ithaca NY, USA 14853}
\email{jg878@cornell.edu}
\thanks{This work was partially supported by National Science Foundation grants DMS-1161175 and DMS-1600635. We thank Richard A.\ Shore for many useful discussions and suggestions. We also thank Paul-Elliot Angles d'Auriac, Takayuki Kihara, Alberto Marcone, Arno Pauly, and Manlio Valenti for their comments and interest.}

\begin{document}

\begin{abstract}
We study the computational content of various theorems with reverse mathematical strength around Arithmetical Transfinite Recursion ($\ATR_0$) from the point of view of computability-theoretic reducibilities, in particular Weihrauch reducibility. Our first main result states that it is equally hard to construct an embedding between two given well-orderings, as it is to construct a Turing jump hierarchy on a given well-ordering. This answers a question of Marcone. We obtain a similar result for Fra\"iss\'e's conjecture restricted to well-orderings. We then turn our attention to K\"onig's duality theorem, which generalizes K\"onig's theorem about matchings and covers to infinite bipartite graphs. Our second main result shows that the problem of constructing a K\"onig cover of a given bipartite graph is roughly as hard as the following ``two-sided'' version of the aforementioned jump hierarchy problem: given a \emph{linear} ordering $L$, construct either a jump hierarchy on $L$ (which may be a pseudohierarchy), or an infinite $L$-descending sequence. We also obtain several results relating the above problems with choice on Baire space (choosing a path on a given ill-founded tree) and unique choice on Baire space (given a tree with a unique path, produce said path).
\end{abstract}

\maketitle

\section{Introduction}

Given any two well-orderings, there must be an embedding from one of the well-orderings into the other. How easy or difficult is it to produce such an embedding? Is this problem more difficult if we are required to produce an embedding whose range forms an initial segment?

Before attempting to answer such questions we ought to discuss how we could formalize them.  One approach is to use well-established notions of (relative) complexity of \emph{sets}. It is ``easy'' to produce an embedding between two given well-orderings, if there is an embedding which is ``simple'' relative to the given well-orderings. Depending on context, ``simple'' could mean computable, polynomial-time computable, etc. On the other hand, one could say it is ``difficult'' to produce an embedding between two given well-orderings, if \emph{any} embedding between them has to be ``complicated'' relative to the given well-orderings. Then we may define a notion of complexity on problems as follows: a problem is ``easy'' if every instance of the problem is ``easy'' in the above sense; a problem is ``difficult'' if there is an instance of the problem which is ``difficult'' in the above sense.

How, then, could we compare the \emph{relative} complexity of such problems? Following the above approach, it is natural to do so by comparing problems against a common yardstick, which is defined using notions of complexity of sets. Computability theory provides several such notions. One example is the number of Turing jumps needed to compute a set, or more generally, its position in the arithmetic hierarchy or the hyperarithmetic hierarchy. Another example is the lowness hierarchy.

This is useful for getting a rough idea of the complexity of a problem, but turns out to be unsuitable for finer calibrations. One reason is that our yardsticks may only be loosely comparable to each other (as is the case for the arithmetic and lowness hierarchies). When comparing two problems, one of them could be simpler from one point of view, but more difficult from another.

Second, even if two problems are equally simple relative to the same yardstick (say, if $X$-computable instances of both problems have $X'$-computable solutions), how do we know if they are related in any sense? Put another way, are they simple for the same ``reason''?

The above considerations suggest a complementary approach: instead of measuring the complexity of problems by measuring the complexity of their solutions, we could focus on the relationships between problems themselves. A common type of ``relationship'' which represents relative complexity is a reduction. Roughly speaking, a problem $P$ is reducible to a problem $Q$ if given an oracle for solving $Q$, we could transform it into an oracle for solving $P$. In order for this notion to be meaningful, such a transformation process has to be simple relative to the difficulty of solving $Q$. In this paper, we will focus on uniformly computable reductions, also known as Weihrauch reductions (Definition \ref{defn:Weihrauch_reducibility}).

Many theorems can be viewed as problems, and for such theorems, a proof of theorem $A$ from theorem $B$ can often be viewed as a reduction from the problem corresponding to theorem $A$ to the problem corresponding to theorem $B$. Therefore, our endeavor of studying reductions between problems is closely related to the program of reverse mathematics, which is concerned with whether a theorem is provable from other theorems (over a weak base theory).

If a proof of theorem $A$ using theorem $B$ does not obviously translate to a reduction from problem $A$ to problem $B$, there are two possible outcomes. Sometimes, we might be able to massage the proof into one that does translate into a reduction. We might also find a different proof of $A$ using $B$ that can be translated into a reduction. Otherwise, we might be able to show that there is no reduction from $A$ to $B$. In that case, this suggests that any proof of $A$ using $B$ has to be somewhat complicated.

Certain questions about the structure of proofs have natural analogs in terms of computable reducibilities. For example, one may appeal to a premise multiple times in the course of a proof. Such appeals may be done in ``parallel'' or in ``series''. One may wonder whether multiple appeals are necessary, or whether appeals in series could be made in parallel instead. These questions can be formalized in the framework of computable reducibilities, for there are ways of combining problems which correspond to applying them in parallel or in series (Definitions \ref{defn:parallel_product}, \ref{defn:compositional_product}).

Finally, the framework of computable reducibilities uncovers and makes explicit various computational connections between problems from computable analysis and theorems that have been studied in reverse mathematics. We will see how the problem of choosing any path on an ill-founded tree and the problem of choosing the path on a tree with a unique path (known as $\C_{\N^\N}$ and $\UC_{\N^\N}$ respectively, see Definition \ref{defn:LPO_and_choice_problems}) are related to theorems which do not obviously have anything to do with trees.

In this paper, we use the framework of computable reducibilities to provide a fine analysis of the computational content of various theorems, such as Fra\"iss\'e's conjecture for well-orderings, weak comparability of well-orderings, and K\"onig's duality theorem for countable bipartite graphs. In reverse mathematics, all of these theorems are known to be equivalent to the system of Arithmetical Transfinite Recursion ($\ATR_0$). Our analysis exposes finer distinctions between these theorems. We describe our main results as follows.

In the first half of this paper, we define a problem $\ATR$ which is analogous to $\ATR_0$ in reverse mathematics (Definition \ref{defn:ATR}). Then we use $\ATR$ to calibrate the computational content of various theorems about embeddings between well-orderings. In particular, we show that:
\begin{quote}
The problem of computing an embedding between two given well-orderings is as hard as $\ATR$ (Theorem \ref{thm:ATR_leq_WCWO}).
\end{quote}
This answers a question of Marcone \cite[Question 5.8]{kmp18}. This also implies that it is no harder to produce an embedding whose range forms an initial segment, than it is to produce an arbitrary embedding. Note that in this case the situation is the same from the point of view of either Weihrauch reducibility or reverse mathematics.

In the second half of this paper, we define several ``two-sided'' problems, which are natural extensions of their ``one-sided'' versions. This allows us to calibrate the computational content of K\"onig's duality theorem for countable bipartite graphs (see section \ref{section:KDT}). In particular, we define a two-sided version of $\ATR$, denoted $\ATR_2$ (Definition \ref{defn:two_sided_ATR}), and show that:

\begin{quote}
The problem of computing a K\"onig cover of a given bipartite graph is (roughly) as hard as $\ATR_2$ (Theorems \ref{thm:ATR_2_leq_c_KDT} and \ref{thm:KDT_arith_W_ATR_2}).
\end{quote}

$\ATR_2$ is much harder than $\ATR$ in terms of computational difficulty (Corollary \ref{cor:two_sided_not_below_UC}), so this example exhibits a marked difference between computable reducibilities and reverse mathematics.

The two-sided problems we study and K\"onig's duality theorem also provide examples of problems which lie strictly between $\UC_{\N^\N}$ and $\C_{\N^\N}$ in the Weihrauch degrees. Other examples exhibiting similar phenomena were studied by Kihara, Marcone, Pauly \cite{kmp18}.

\section{Background}

\subsection{Computability}

For background on hyperarithmetic theory, we refer the reader to Sacks \cite[I--III]{sacks}. We will use the following version of ``effective transfinite recursion'' on linear orderings, which easily follows from the recursion theorem.

\begin{thm}
Let $L$ be an $X$-computable linear ordering. Suppose $F: \N \to \N$ is total $X$-computable and for all $e \in \N$ and $b \in L$, if $\Phi^X_e(a)\conv$ for all $a <_L b$, then $\Phi^X_{F(e)}(b)\conv$. Then there is some $e$ such that $\Phi^X_e \simeq \Phi^X_{F(e)}$. Furthermore: 
\begin{itemize}
	\item $\{b: \Phi^X_e(b)\dv\}$ is either empty or contains an infinite $<_L$-descending sequence;
	\item Such an index $e$ can be found uniformly in $X$, an index for $F$, and an index for $L$.
\end{itemize}
\end{thm}

In many of our applications, $X$ will be a sequence of sets $\langle X_a \rangle_a$ indexed by elements of a linear ordering (sometimes $L$, but not always). We will think of $\Phi^X_e$ as a partial function $f: L \to \N$, and we will think of each $f(b)$ as an index for a computation from some $X_a$.

\subsection{Representations} \label{subsection:repns}

Let $X$ be a set of countable structures, such as (countable) linear orderings, trees, or graphs. A \emph{($\N^\N$-)representation} of $X$ allows us to transfer notions of computability from $\N^\N$ to $X$. Formally, a representation of $X$ is a surjective (possibly partial) map $\delta: \subseteq \N^\N \to X$. (More generally, $X$ can be any set of cardinality at most that of $\N^\N$.)

The pair $(X,\delta)$ is called a \emph{represented space}. If $\delta(p) = x$ then we say that $p$ is a \emph{($\delta$-)name} for $x$. Every $x \in X$ has at least one $\delta$-name. We say that $x \in X$ is \emph{computable} if it has some $\delta$-name which is computable.

If we have two representations $\delta$ and $\delta'$ of a set $X$, we say that $\delta$ is \emph{computably reducible} to $\delta'$ if there is some computable function $F: \subseteq \N^\N \to \N^\N$ such that for all $p \in \dom(\delta)$, $\delta(p) = \delta'(F(p))$. We say $\delta$ and $\delta'$ are \emph{computably equivalent} if they are computably reducible to each other. Computably equivalent representations of $X$ induce the same notion of computability on $X$.

Typically, the spaces $X$ we work with have a standard representation (or encoding), which we will not specify in detail. We will work extensively with the represented spaces of linear orderings and well-orderings, so we describe their representations as follows. If $L$ is a linear ordering or well-ordering whose domain is a subset of $\N$, we represent it as the relation $\{\langle a,b \rangle: a \leq_L b\}$. Then the following operations are computable:
\begin{itemize}
	\item checking if a given element is in the domain of the ordering;
	\item adding two given orderings (denoted by $+$);
	\item adding a given sequence of orderings (denoted by $\Sigma$);
	\item multiplying two given orderings (denoted by $\cdot$);
	\item restricting a given ordering to a given subset of its domain.
\end{itemize}
On the other hand, the following operations are not computable:
\begin{itemize}
	\item checking whether a given element is a successor or limit;
	\item finding the successor of a given element (if it exists);
	\item comparing the ordertype of two given well-orderings;
	\item checking if a given real is a name for a well-ordering.
\end{itemize}

In section \ref{section:KDT}, we will work with \emph{rooted subtrees} of $\N^{<\N}$, which are subsets $T$ of $\N^{<\N}$ for which there is a unique $r \in T$ (called the \emph{root}) such that:
\begin{itemize}
	\item no proper prefixes of $r$ lie in $T$;
	\item for every $s \in T$, $s$ extends $r$ and every prefix of $s$ which extends $r$ lies in $T$.
\end{itemize}
A rooted subtree of $\N^{<\N}$ whose root is the empty node $\langle \rangle$ is just a prefix-closed subset of $\N^{<\N}$.

If $r \in \N^{<\N}$ and $R \subseteq \N^{<\N}$, we define $r\concat R = \{r\concat s: s \in R\}$. In particular, if $T \subseteq \N^{<\N}$ is prefix-closed, then $r\concat T$ is a subtree of $\N^{<\N}$ with root $r$. Conversely, if a rooted subtree of $\N^{<\N}$ has root $r$, it is equal to $r\concat T$ for some such $T$. If $T$ is prefix-closed, we sometimes refer to a tree of the form $r\concat T$ as a \emph{copy} of $T$. (Our usage of ``copy'' is more restrictive than its usage in computable structure theory.)

If $T$ is a rooted subtree of $\N^{<\N}$, for any $t \in T$, the \emph{subtree of $T$ above $t$} is the subtree $\{s \in T: t \preceq s\}$ with root $t$.

For each $r \in \N^{<\N}$, $e \in \N$ and $X \subseteq \N$, $(r,e,X)$ is a name for the following tree $T$ with root node $r$: $r\concat\sigma \in T$ if and only if for all $k < |\sigma|$, $\Phi^X_{e,\prod_{i<k} (\sigma(i)+1)}(\sigma\restriction k)\conv = 1$. This representation is easily seen to be computably equivalent to what is perhaps the usual representation, where if $\Phi^X_e$ is total, then $(r,e,X)$ is the name for the tree defined by $\Phi^X_e$ starting with root $r$. The advantage of our representation is that $(r,e,X)$ names some tree even if $\Phi^X_e$ is partial, which will be useful when $e$ is produced by the recursion theorem.

Using the above representation, we can define a representation for sequences of subtrees of $\N^{<\N}$: view $(e,X)$ as $\langle (\langle n \rangle,e_n,X) \rangle_n$, where $e_n$ is an $X$-index for $\Phi^X_e(n,\cdot)$. Observe that every $(e,X)$ names some such sequence.

We will also work with bipartite graphs in section \ref{section:KDT}. We represent bipartite graphs as their vertex set and edge relation. Alternatively, our representation of a bipartite graph could also include a partition of its vertex set which witnesses that the graph is bipartite. Even though these two representations are not computably equivalent\footnote{In fact, there is a computable bipartite graph such that no computable partition of its vertices witnesses that the graph is bipartite. This was known to Bean \cite[remarks after Theorem 7]{bean76} (we thank Jeff Hirst for pointing this out.) See also Hirst \cite[Corollary 3.17]{hirst_thesis}.}, all of our results hold for either representation.

\subsection{Weihrauch reducibility and computable reducibility}

For a self-contained introduction to Weihrauch reducibility, we refer the reader to Brattka, Gherardi, Pauly \cite{bgp18}. In this section, we will only present the definitions that we need in this paper.

We begin by identifying problems, such as that of constructing an embedding between two given well-orderings, with (possibly partial) multivalued functions between represented spaces, denoted $P: \subseteq X \rightrightarrows Y$. A theorem of the form
\[ (\forall x \in X)(\Theta(x) \rightarrow (\exists y \in Y)\Psi(x,y)) \]
corresponds to the multivalued function $P: \subseteq X \rightrightarrows Y$ where $P(x) = \{y \in Y: \Psi(x,y)\}$. Note that logically equivalent statements can correspond to different problems.

The \emph{domain} of a problem, denoted $\dom(P)$, is the set of $x \in X$ such that $P(x)$ is nonempty. An element of the domain of $P$ is called a \emph{$P$-instance}. If $x$ is a $P$-instance, an element of $P(x)$ is called a \emph{$P$-solution to $x$}.

A \emph{realizer} of a problem $P$ is a (single-valued) function $F: \subseteq \N^\N \to \N^\N$ which takes any name for a $P$-instance to a name for any of its $P$-solutions. Intuitively, $P$ is reducible to $Q$ if one can transform any realizer for $Q$ into some realizer for $P$. If such a transformation can be done in a uniformly computable way, then $P$ is said to be Weihrauch reducible to $Q$:

\begin{defn} \label{defn:Weihrauch_reducibility}
$P$ is \emph{Weihrauch reducible} (or \emph{uniformly reducible}) to $Q$, written $P \leq_W Q$, if there are computable functions $\Phi,\Psi: \subseteq \N^\N \to \N^\N$ such that:
\begin{itemize}
	\item given a name $p$ for a $P$-instance, $\Phi(p)$ is a name for a $Q$-instance;
	\item given a name $q$ for a $Q$-solution to the $Q$-instance named by $\Phi(p)$, $\Psi(p \oplus q)$ is a name for a $P$-solution to the $P$-instance named by $p$.
\end{itemize}
In this case, we say that $\Phi$ and $\Psi$ are \emph{forward} and \emph{backward} functionals, respectively, for a Weihrauch reduction from $P$ to $Q$.

We say that $P$ is \emph{arithmetically Weihrauch reducible} to $Q$, written $P \leq_W^{\arith} Q$, if the above holds for some arithmetically defined functions $\Phi$ and $\Psi$, or equivalently, some computable functions $\Phi$ and $\Psi$ which are allowed access to some fixed finite Turing jump of their inputs.
\end{defn}

For readability, we will typically not mention names in our proofs. For example, we will write ``given a $P$-instance'' instead of ``given a name for a $P$-instance''.

It is easy to see that Weihrauch reducibility is reflexive and transitive, and hence defines a degree structure on problems. In fact, there are several other natural operations on problems that define corresponding operations on the Weihrauch degrees. In the following, we define only the operations that we use.

First we have the parallel product, which captures the power of applying problems in parallel:

\begin{defn} \label{defn:parallel_product}
The \emph{parallel product} of $P$ and $Q$, written $P \times Q$, is defined as follows: $\dom(P \times Q) = \dom(P) \times \dom(Q)$ and $(P \times Q)(x,y) = P(x) \times Q(y)$. The (infinite) \emph{parallelization} of $P$, written $\widehat{P}$, is defined as follows: $\dom(\widehat{P}) = \dom(P)^\N$ and $\widehat{P}((x_n)_n) = \{(y_n)_n: y_n \in P(x_n)\}$.
\end{defn}

It is easy to see that the parallel product and parallelization of problems induce corresponding operations on their Weihrauch degrees.

More generally, we can also apply problems in series:

\begin{defn}
The \emph{composition} $\circ$ is defined as follows: for $P: \subseteq X \rightrightarrows Y$ and $Q: \subseteq Y \rightrightarrows Z$, we define $\dom(Q \circ P) = \{x \in X: P(x) \subseteq \dom(Q)\}$ and $(Q \circ P)(x) = \{z \in Z: \exists y \in P(x)(z \in Q(y))\}$.
\end{defn}

The composition of problems, however, does not directly induce a corresponding operation on Weihrauch degrees. It is also too restrictive, in the sense that a $P$-solution is required to be literally a $Q$-instance. Nevertheless, one can use the composition to define an operation on Weihrauch degrees that more accurately captures the power of applying two problems in series:

\begin{defn}[Brattka, Gherardi, Marcone \cite{bgm12}] \label{defn:compositional_product}
The \emph{compositional product} $\ast$ is defined as follows:
\[ Q \ast P = \sup \{Q_0 \circ P_0: Q_0 \leq_W Q, P_0 \leq_W P\}, \]
where the $\sup$ is taken over the Weihrauch degrees.
\end{defn}

Brattka and Pauly \cite{bp16} showed that $Q \ast P$ always exists.

Next, we define some well-studied problems that are helpful for calibrating the problems we are interested in.

\begin{defn} \label{defn:LPO_and_choice_problems}
Define the following problems:
\begin{itemize}
	\item[$\LPO$:] given $p \in \N^\N$, output $1$ if there is some $k \in \N$ such that $p(k) = 0$, else output $0$;
	\item[$\C_\N$:] given some $f: \N \to \N$ which is not surjective, output any $x$ not in the range of $f$;
	\item[$\C_{\N^\N}$:] given an ill-founded subtree of $\N^{<\N}$, output any path on it;
	\item[$\UC_{\N^\N}$:] given an ill-founded subtree of $\N^{<\N}$ with a unique path, output said path.
\end{itemize}
\end{defn}

For more information about the above problems, we refer the reader to the survey by Brattka, Gherardi, Pauly \cite{bgp18}.

Finally, we define a non-uniform coarsening of Weihrauch reducibility known as computable reducibility.

\begin{defn}[Dzhafarov \cite{d15}]
$P$ is \emph{computably reducible} to $Q$, written $P \leq_c Q$, if given a name $p$ for a $P$-instance, one can compute a name $p'$ for a $Q$-instance such that given a name $q$ for a $Q$-solution to the $Q$-instance named by $p'$, one can use $p \oplus q$ to compute a name for a $P$-solution to the $P$-instance named by $p$.
\end{defn}

For example, even though $\LPO$ is not Weihrauch reducible to the identity function, it is computably reducible to the identity because a solution to an $\LPO$-instance is either $0$ or $1$. Observe that $\LPO$ is also arithmetically Weihrauch reducible to the identity. The same conclusions hold for $\C_\N$.

The following easy proposition will help us derive corollaries of our results which involve computable reducibility and arithmetic Weihrauch reducibility:

\begin{prop}
Suppose $R \leq_W Q \ast P$. If $Q \leq_c \id$, then $R \leq_c P$. If $Q \leq_W^{\arith} \id$, then $R \leq_W^{\arith} P$.
\end{prop}

\section{An $\ATR$-like problem} \label{section:ATR}

In this section, we formulate a problem which is analogous to $\ATR_0$ in reverse mathematics. Informally, $\ATR_0$ in reverse mathematics asserts that one can iterate the Turing jump along any countable well-ordering starting at any set \cite[pg.\ 38]{sim_book}. We make that precise as follows:

\setcounter{sectiontmp1}{\value{section}}
\setcounter{thmtmp1}{\value{thm}}

\begin{defn} \label{defn:jump_hierarchy}
Let $L$ be a linear ordering with first element $0_L$, and let $A \subseteq \N$. We say that $\langle X_a \rangle_{a \in L}$ is a \emph{jump hierarchy on $L$ which starts with $A$} if $X_0 = A$ and for all $b >_L 0_L$, $X_b = (\bigoplus_{a <_L b} X_a)'$.
\end{defn}

There are several ways to define jump hierarchies. We have chosen the above definition for our convenience. We will show that the Weihrauch degree of the resulting problem is rather robust with regards to which definition we choose. See, for example, Proposition \ref{prop:ATR_labeled}.

Note that by transfinite recursion and transfinite induction, for any well-ordering $L$ and any set $A$, there is a unique jump hierarchy on $L$ which starts with $A$.

\begin{defn} \label{defn:ATR}
Define the problem $\ATR$ as follows. Instances are pairs $(L,A)$ where $L$ is a well-ordering and $A \subseteq \N$, with unique solution being the jump hierarchy $\langle X_a \rangle_{a \in L}$ which starts with $A$.
\end{defn}

There is a significant difference between the problem $\ATR$ and the system $\ATR_0$ in reverse mathematics, as expounded in the remark after Theorem 3.2 in Kihara, Marcone, Pauly \cite{kmp18}. For example, in the setting of reverse mathematics, different models may disagree on which linear orderings are well-orderings.

The standard definition of $\ATR_0$ in reverse mathematics \cite[Definition V.2.4]{sim_book} involves iterating arbitrary arithmetical operators instead of just the Turing jump. We formulate that statement as a problem and show that it is Weihrauch equivalent to $\ATR$.

\begin{prop} \label{prop:ATR_arith_formula}
$\ATR$ is Weihrauch equivalent to the following problem. Instances are triples $(L,A,\Theta)$ where $L$ is a well-ordering, $A \subseteq \N$, and $\Theta(n,Y,A)$ is an arithmetical formula whose only free variables are $n$, $Y$ and $A$, with unique solution $\langle Y_a \rangle_{a \in L}$ such that for all $b \in L$, $Y_b = \{n: \Theta(n,\bigoplus_{a <_L b} Y_a,A)\}$.
\end{prop}
\begin{proof}
$\ATR$ is Weihrauch reducible to the above problem: for the forward reduction, given $(L,A)$, consider $(L,A,\Theta)$ where $\Theta(n,Y,A)$ holds if either $Y = \emptyset$ and $n \in A$, or $n \in Y'$. The backward reduction is the identity.

Conversely, given $(L,A,\Theta)$, let $k$ be one greater than the number of quantifier alternations in $\Theta$. Apply $\ATR$ to $(1 + k \cdot L,L \oplus A)$ to obtain the jump hierarchy $\langle X_\alpha \rangle_{\alpha \in 1+k\cdot L}$.

For the backward reduction, we will use $\langle X_{(a,k-1)} \rangle_{a \in L}$-effective transfinite recursion along $L$ to define a total $\langle X_{(a,k-1)} \rangle_{a \in L}$-recursive function $f: L \to \N$ such that:
\begin{itemize}
	\item $\Phi^{X_{(b,k-1)}}_{f(b)}$ is total for all $b \in L$;
	\item if we define $Y_b = \Phi^{X_{(b,k-1)}}_{f(b)}$ for all $b \in L$, then $Y_b = \{n: \Theta(n,\bigoplus_{a <_L b} Y_a,A)\}$.
\end{itemize}
For each $b \in L$, we define $\Phi^{X_{(b,k-1)}}_{f(b)}$ as follows. First note that $X_{(b,0)}$ uniformly computes $L \oplus A$ (because of the $1$ in front of $1+k\cdot L$), and hence uniformly computes $A \oplus \bigoplus_{a <_L b} X_{(a,k-1)}$. Now $X_{(b,k-1)}$ uniformly computes $X_{(b,0)}^{(k)}$, which uniformly computes $\left(A \oplus \bigoplus_{a <_L b} X_{(a,k-1)}\right)^{(k)}$. Since $\Phi^{X_{(a,k-1)}}_{f(a)}$ is total for all $a <_L b$, that in turn uniformly computes $\left(A \oplus \bigoplus_{a <_L b} Y_a\right)^{(k)}$, where $Y_a$ is defined to be $\{n: \Phi^{X_{(a,k-1)}}_{f(a)}(n)\conv=1\}$. Finally, $\left(A \oplus \bigoplus_{a <_L b} Y_a\right)^{(k)}$ uniformly computes $\{n: \Theta(n,\bigoplus_{a <_L b} Y_a,A)\}$, which defines $\Phi^{X_{(b,k-1)}}_{f(b)}$ as desired.

By transfinite induction along $L$, $f$ is total. Hence we can compute $Y_b = \Phi^{X_{(b,k-1)}}_{f(b)}$ for all $b \in L$, and output $\langle Y_b \rangle_{b \in L}$.
\end{proof}

When we define reductions from $\ATR$ to other problems by effective transfinite recursion, we will often want to perform different actions at the first step, successor steps, and limit steps. If we want said reductions to be uniform, we want to be able to compute which step we are in. This motivates the following definition:

\begin{defn}
A \emph{labeled well-ordering} is a tuple $\L = (L,0_L,S,p)$ where $L$ is a well-ordering, $0_L$ is the first element of $L$, $S$ is the set of all successor elements in $L$, and $p: S \to L$ is the predecessor function.
\end{defn}

We show that when defining Weihrauch reductions from $\ATR$ to other problems, we may assume that the given well-ordering has labels:

\begin{prop} \label{prop:ATR_equiv_ATR_labels}
$\ATR$ is Weihrauch equivalent to the following problem. Instances are pairs $(\L,A)$ where $\L = (L,0_L,S,p)$ is a labeled well-ordering and $A \subseteq \N$, with unique solution being the jump hierarchy $\langle X_a \rangle_{a \in L}$ which starts with $A$.
\end{prop}
\begin{proof}
Given $(L,A)$, we can uniformly compute labels for $\omega \cdot (1+L)$. Then apply the above problem to $(\omega \cdot (1+L),L \oplus A)$ to obtain the jump hierarchy $\langle X_{(n,\alpha)} \rangle_{n \in \omega,\alpha \in 1+L}$ which starts with $L \oplus A$.

For the backward reduction, we will use $\langle X_{(0,b)} \rangle_{b \in L}$-effective transfinite recursion along $L$ to define a total $\langle X_{(0,b)} \rangle_{b \in L}$-recursive function $f: L \to \N$ such that $\Phi^{X_{(0,b)}}_{f(b)}$ is total for every $b \in L$ and $\langle \Phi^{X_{(0,b)}}_{f(b)} \rangle_{b \in L}$ is the jump hierarchy on $L$ which starts with $A$.

First note that every $X_{(0,b)}$ uniformly computes $(L \oplus A)'$, and hence $0_L$.  This means that it uniformly computes the case division in the following construction.

For the base case, $X_{(0,0_L)}$ uniformly computes $L \oplus A$ and hence $A$. As for $b >_L 0_L$, $X_{(0,b)}$ uniformly computes $L$, hence it uniformly computes $(\bigoplus_{a <_L b} X_{(0,a)})'$. Therefore it uniformly computes $(\bigoplus_{a <_L b} \Phi^{X_{(0,a)}}_{f(a)})'$.
\end{proof}

The following closure property will be useful for proving Proposition \ref{prop:WQO_WO_ATR_CWO_ATR}. This fact also follows from the combination of work of Pauly ($\UC_{\N^\N}$ is parallelizable \cite{pauly15}) and Kihara, Marcone, Pauly ($\ATR \equiv_W \UC_{\N^\N}$ \cite{kmp18}), but we provide a short direct proof.

\begin{prop} \label{prop:ATR_parallelizable}
$\ATR$ is parallelizable, i.e., $\widehat{\ATR} \equiv_W \ATR$.
\end{prop}
\begin{proof}
It suffices to show that $\widehat{\ATR} \leq_W \ATR$. Instead of $\widehat{\ATR}$, we consider the parallelization of the version of $\ATR$ in Proposition \ref{prop:ATR_equiv_ATR_labels}. Given $\langle (\L_i,A_i) \rangle_i$, apply $\ATR$ to $(\sum_i L_i,\bigoplus_i L_i \oplus A_i)$ to obtain the jump hierarchy $\langle X_{(i,a)} \rangle_{i \in \omega,a \in L_i}$ which starts with $\bigoplus_i L_i \oplus A_i$.

For each $i$, we show how to compute the jump hierarchy $\langle X_a \rangle_{a \in L_i}$ which starts with $A_i$ using $(\L_0 \oplus \L_i \oplus \langle X_{(i,a)} \rangle_{a \in L_i})$-effective transfinite recursion along $L_i$. This is done by defining a total $(\L_0 \oplus \L_i \oplus \langle X_{(i,a)} \rangle_{a \in L_i})$-recursive function $f_i: L_i \to \N$ such that for all $a \in L_i$, $\Phi^{X_{(i,a)}}_{f(a)}$ is total and defines $X_a$. (The role of $\L_0 \oplus \L_i$ is to provide the values of $0_{L_0}$ and $0_{L_i}$ in the following computation.)

For the base case, $X_{(i,0_{L_i})}$ uniformly computes $X_{(0,0_{L_0})} = \bigoplus_i L_i \oplus A_i$, which uniformly computes $A_i$. 

For $b >_{L_i} 0_{L_i}$, $X_{(i,b)}$ uniformly computes $X_{(0,0_{L_0})}$ which uniformly computes $L_i$, so $X_{(i,b)}$ uniformly computes $(\bigoplus_{a <_{L_i} b} X_{(i,a)})'$. That in turn uniformly computes $(\bigoplus_{a <_{L_i} b} \Phi^{X_{(i,a)}}_{f(a)})' = (\bigoplus_{a <_{L_i} b} X_a)' = X_b$ as desired.
\end{proof}

Henceforth we will primarily work with the following version of $\ATR$:

\begin{prop} \label{prop:ATR_labeled}
$\ATR$ is Weihrauch equivalent to the following problem: instances are pairs $(\L,c)$ where $\L$ is a labeled well-ordering and $c \in L$, with unique solution being $Y_c$, where $\langle Y_a \rangle_{a \in L}$ is the unique hierarchy such that:
\begin{itemize}
	\item $Y_{0_L} = \L$;
	\item if $b$ is the successor of $a$, then $Y_b = Y'_a$;
	\item if $b$ is a limit, then $Y_b = \bigoplus_{a <_L b} Y_a$.
\end{itemize}
\end{prop}
\begin{proof}
Using Proposition \ref{prop:ATR_arith_formula}, it is easy to see that the above problem is Weihrauch reducible to $\ATR$.

Conversely, we reduce the version of $\ATR$ in Proposition \ref{prop:ATR_equiv_ATR_labels} to the above problem. Given $(\L,A)$, define
\[ M = \omega \cdot (1+(A,<_\N)+L+1)+1. \]
Formally, the domain of $M$ is
\begin{align*}
\{(0,n): n \in \omega\} \cup \{(1,m,n): m\in A,n \in \omega\} \\
\cup \{(2,a,n): a \in L,n \in \omega\} \cup \{(3,n): n \in \omega\} \cup \{m_M\}
\end{align*}
with the ordering described above. It is easy to see that $L \oplus A$ uniformly computes $M$ and labels for it. Let $\M$ denote the tuple of $M$ and its labels.

Apply the given problem to $\M$ and $m_M \in M$ to obtain $Y_{m_M}$. Note that since $m_M$ is a limit, $Y_{m_M}$ uniformly computes $Y_{(0,0)} = \M$, and hence $\langle Y_c \rangle_{c \in M}$.

For the backward functional, we perform $(\L \oplus \langle Y_c \rangle_{c \in M})$-effective transfinite recursion along $L$ to define a total $(\L \oplus \langle Y_c \rangle_{c \in M})$-recursive function $f: L \to \N$ such that for each $a \in L$, $\Phi^{Y_{(2,a,1)}}_{f(a)}$ is total and defines the $a^{\text{th}}$ column $X_a$ of the jump hierarchy on $L$ which starts with $A$. Note that $\L$ uniformly computes the following case division.

For the base case, first use $Y_{(2,0_L,1)} = Y'_{(2,0_L,0)}$ to compute $Y_{(2,0_L,0)}$. Now $(2,0_L,0)$ is a limit, so $Y_{(2,0_L,0)}$ uniformly computes $Y_{(0,0)} = \M$, which uniformly computes $A$ as desired.

For $b >_L 0_L$, since $(2,b,0)$ is a limit, $Y_{(2,b,0)}$ uniformly computes $Y_{(0,0)} = \M$, which uniformly computes $L$. Therefore $Y_{(2,b,0)}$ uniformly computes $\bigoplus_{a <_L b} Y_{(2,a,1)}$, and hence $\bigoplus_{a <_L b} \Phi^{Y_{(2,a,1)}}_{f(a)} = \bigoplus_{a <_L b} X_a$. Therefore $Y_{(2,b,1)}$ uniformly computes $X_b = (\bigoplus_{a <_L b} X_a)'$ as desired.

This completes the definition of $f$, and hence the reduction from the version of $\ATR$ in Proposition \ref{prop:ATR_equiv_ATR_labels} to the given problem.
\end{proof}

Thus far, we have seen that the Weihrauch degree of $\ATR$ is fairly robust with respect to the type of jump hierarchy that it outputs (Propositions \ref{prop:ATR_arith_formula}, \ref{prop:ATR_equiv_ATR_labels}, \ref{prop:ATR_labeled}). However, we still require some level of uniformity in the jump hierarchy produced:

\begin{prop} \label{prop:lim_not_leq_W_nonuniform_jump_hierarchy}
The problem of producing the Turing jump of a given set is not Weihrauch reducible to the following problem: instances are pairs $(L,A)$ where $L$ is a well-ordering and $A \subseteq \N$, and solutions to $L$ are hierarchies $\langle X_a \rangle_{a \in L}$ where $X_{0_L} = A$ and for all $a <_L b$, $X'_a \leq_T X_b$. Hence $\ATR$ is not Weihrauch reducible to the latter problem either.
\end{prop}
\begin{proof}
Towards a contradiction, fix forward and backward Turing functionals $\Gamma$ and $\Delta$ witnessing otherwise. We will show that $\Gamma$ and $\Delta$ could fail to produce $\emptyset'$ from $\emptyset$. First, $\Gamma^\emptyset$ defines some computable $(L,A)$. We claim that there are finite $\langle \sigma_a \rangle_{a \in L}$ and $e$ such that $\sigma_{0_L} \prec A$ and $\Delta^{\emptyset \oplus \langle \sigma_a \rangle_{a \in L}}(e)\conv \neq \emptyset'(e)$.

Suppose not. Then for each $e$, we may compute $\emptyset'(e)$ by searching for $\langle \sigma_a \rangle_{a \in L}$ such that $\sigma_{0_L} \prec A$ and $\Delta^{\emptyset \oplus \langle \sigma_a \rangle_{a \in L}}(e)\conv$. Such $\langle \sigma_a \rangle_{a \in L}$ must exist because if $\langle X_a \rangle_{a \in L}$ is a hierarchy on $L$ which starts with $A$ (as defined in the proposition), then $\Delta^{\emptyset \oplus \langle X_a \rangle_{a \in L}}$ is total. This is a contradiction, thereby proving the claim.

Fix any $\langle \sigma_a \rangle_{a \in L}$ which satisfies the claim. It is clear that $\langle \sigma_a \rangle_{a \in L}$ can be extended to a solution $\langle X_a \rangle_{a \in L}$ to $(L,A)$ for the given problem (e.g., by extending using columns of the usual jump hierarchy). But $\Delta^{\emptyset \oplus \langle X_a \rangle_{a \in L}} \neq \emptyset'$, contradiction.
\end{proof}

If we are willing to allow arithmetic Weihrauch reductions, then $\ATR$ remains robust:

\begin{prop}
$\ATR$ is arithmetically Weihrauch reducible (hence arithmetically Weihrauch equivalent) to the problem in Proposition \ref{prop:lim_not_leq_W_nonuniform_jump_hierarchy}.
\end{prop}

For the proof, we refer to the reader to the proof of Proposition \ref{prop:ATR2_nonuniform_jump_hierarchy} later. (The only difference is that we use transfinite induction along the given well-ordering to show that we always output a jump hierarchy.)

\section{Theorems about embeddings between well-orderings}

There are several theorems about embeddings between well-orderings which lie around $\ATR_0$ in reverse mathematics. Friedman (see \cite[notes for Theorem V.6.8, pg.\ 199]{sim_book}) showed that comparability of well-orderings is equivalent to $\ATR_0$. Friedman and Hirst \cite{fh90} then showed that weak comparability of well-orderings is also equivalent to $\ATR_0$. We formulate those two theorems about embeddings as problems:

\begin{defn}
Define the following problems:
\begin{itemize}
	\item[$\CWO$:] Given a pair of well-orderings, produce an embedding from one of them onto an initial segment of the other.
	\item[$\WCWO$:] Given a pair of well-orderings, produce an embedding from one of them into the other.
\end{itemize}
\end{defn}

Marcone proved the analog of Friedman's result for Weihrauch reducibility:

\begin{thm}[see Kihara, Marcone, Pauly \cite{kmp18}] \label{thm:marcone_CWO_equiv_ATR}
$\CWO \equiv_W \UC_{\N^\N} \equiv_W \ATR$.
\end{thm}

(In fact, he proved the equivalence up to strong Weihrauch reducibility, which we will not define here.) In Theorem \ref{thm:ATR_leq_WCWO}, we prove the analog of Friedman and Hirst's result for Weihrauch reducibility, i.e., $\WCWO \equiv_W \UC_{\N^\N}$. This answers a question of Marcone \cite[Question 5.8]{kmp18}.

Another class of examples of theorems about embeddings comes from Fra\"iss\'e's conjecture (proved by Laver \cite{laver_fraisse}), which asserts that the set of countable linear orderings is well-quasi-ordered (i.e., any infinite sequence contains a weakly increasing pair) by embeddability. Shore \cite{shore93} studied the reverse mathematics of various restrictions of Fra\"iss\'e's conjecture. We formulate them as problems:

\begin{defn}
Define the following problems:
\begin{itemize}
	\item[$\WQO_{\LO}$:] Given a sequence $\langle L_i \rangle$ of linear orderings, produce $i < j$ and an embedding from $L_i$ into $L_j$.
	\item[$\WQO_{\WO}$:] Given a sequence $\langle L_i \rangle$ of well-orderings, produce $i < j$ and an embedding from $L_i$ into $L_j$.
	\item[$\NDS_{\WO}$:] Given a sequence $\langle L_i \rangle$ of well-orderings, and embeddings $\langle F_i \rangle$ from each $L_{i+1}$ into $L_i$, produce $i < j$ and an embedding from $L_i$ into $L_j$.
	\item[$\NIAC_{\WO}$:] Given a sequence $\langle L_i \rangle$ of well-orderings, produce $i$ and $j$ (we may have $i > j$) and an embedding from $L_i$ into $L_j$.
\end{itemize}
$\NDS_{\LO}$ and $\NIAC_{\LO}$ can be defined analogously, but we will not study them in this paper.
\end{defn}

$\WQO_{\LO}$ corresponds to Fra\"iss\'e's conjecture. $\WQO_{\WO}$ is the restriction of Fra\"iss\'e's conjecture to well-orderings. $\NDS_{\WO}$ asserts that there is no infinite strictly descending sequence of well-orderings. $\NIAC_{\WO}$ asserts that there is no infinite antichain of well-orderings.

The definitions immediately imply that

\begin{prop} \label{prop:multivalued_functions_embeddings_basic_relationships}
\begin{gather*}
\NDS_{\WO} \leq_W \WQO_{\WO} \leq_W \WQO_{\LO} \\
\NIAC_{\WO} \leq_W \WCWO \leq_W \CWO \\
\NIAC_{\WO} \leq_W \WQO_{\WO}
\end{gather*}
\end{prop}

It is not hard to show that all of the problems in Proposition \ref{prop:multivalued_functions_embeddings_basic_relationships}, except for $\WQO_{\LO}$, are Weihrauch reducible to $\ATR$. (We bound the strength of $\WQO_\LO$ in Corollaries \ref{cor:C_N_N_not_below_two_sided_problems} and \ref{cor:two_sided_not_below_UC}.)

\begin{prop} \label{prop:WQO_WO_ATR_CWO_ATR}
$\CWO \leq_W \ATR$ and $\WQO_{\WO} \leq_W \ATR$.
\end{prop}
\begin{proof}
Let $Q$ denote the following apparent strengthening of $\CWO$: a $Q$-instance is a pair of well-orderings $(L,M)$, and a $Q$-solution consists of both a $\CWO$-solution $F$ to $(L,M)$ and an indication of whether $L < M$, $L \equiv M$, or $L > M$. Clearly $\CWO \leq_W Q$. (Marcone showed that $\CWO \equiv_W \ATR$ (Theorem \ref{thm:marcone_CWO_equiv_ATR}), so actually $\CWO \equiv_W Q$.)

We start by showing that $Q \leq_W \ATR$. Given $(L,M)$, define $N$ by adding a first element $0_N$ and a last element $m_N$ to $L$. Apply the version of $\ATR$ in Proposition \ref{prop:ATR_arith_formula} to obtain a hierarchy $\langle X_a \rangle_{a \in N}$ such that:
\begin{itemize}
	\item $X_{0_N} = L \oplus M$;
	\item for all $b >_N 0_N$, $X_b = \left(\bigoplus_{a <_N b} X_a\right)'''$. 
\end{itemize}

For the backward reduction, we start by using $\langle X_a\rangle_{a \in L}$-effective transfinite recursion along $L$ to define a total $\langle X_a\rangle_{a \in L}$-recursive function $f: L \to \N$ such that $\{(a,\Phi^{X_a}_{f(a)}(0)) \in L \times M: \Phi^{X_a}_{f(a)}(0)\conv\}$ is an embedding of an initial segment of $L$ into an initial segment of $M$.

To define $f$, if we are given any $b \in L$ and $f\restriction \{a: a <_L b\}$, we need to define $f(b)$, specifically $\Phi^{X_b}_{f(b)}(0)$. Use $X_b = (\bigoplus_{a <_L b} X_a)'''$ to compute whether there is an $M$-least element above $\{\Phi^{X_a}_{f(a)}(0): a <_L b\}$ (equivalently, whether $M\backslash \{\Phi^{X_a}_{f(a)}(0): a <_L b\}$ is nonempty). If so, we output said $M$-least element; otherwise diverge. This completes the definition of $\Phi^{X_b}_{f(b)}(0)$.

Apply the recursion theorem to the definition above to obtain a partial $\langle X_a\rangle_{a \in L}$-recursive function $f: L \to \N$. Now, to complete the definition of the backward reduction we consider the following cases.

\underline{Case 1.} $f$ is total. Then we output $\{(a,\Phi^{X_a}_{f(a)}(0)): a \in L\}$, which is an embedding from $L$ onto an initial segment of $M$.

\underline{Case 2.} Otherwise, $\{\Phi^{X_a}_{f(a)}(0): a \in L, \Phi^{X_a}_{f(a)}(0)\conv\} = M$. Then we output $\{(\Phi^{X_a}_{f(a)}(0),a): a \in L, \Phi^{X_a}_{f(a)}(0)\conv\}$, which is an embedding from $M$ onto an initial segment of $L$.

Finally, note that the last column $X_{m_N}$ of $\langle X_a \rangle_{a \in N}$ can compute which case holds and compute the appropriate output for each case. If Case 1 holds but not Case 2, then $L < M$. If Case 2 holds but not Case 1, then $L > M$. If both Case 1 and 2 hold, then $L \equiv M$.

Next, we turn our attention to $\WQO_\WO$. Observe that $\WQO_{\WO} \leq_W \widehat{Q}$: given a sequence $\langle L_i \rangle$ of well-orderings, apply $Q$ to each pair $(L_i,L_j)$, $i < j$. Search for the least $(i,j)$ such that $Q$ provides an embedding from $L_i$ into $L_j$, and output accordingly.

Finally, $\widehat{Q} \leq_W \widehat{\ATR} \equiv_W \ATR$ (Proposition \ref{prop:ATR_parallelizable}), so $\WQO_{\WO} \leq_W \ATR$ as desired.
\end{proof}

In the next few sections, we work toward some reversals. Central to a reversal (say, from $\WCWO$ to $\ATR$) is the ability to encode information into well-orderings such that we can extract information from an arbitrary embedding between them. Shore \cite{shore93} showed how to do this if the well-orderings are indecomposable (and constructed appropriately).

\begin{defn}
A well-ordering $X$ is \emph{indecomposable} if it is embeddable in all of its final segments.
\end{defn}

Indecomposable well-orderings also played an essential role in Friedman and Hirst's \cite{fh90} proof that $\WCWO$ implies $\ATR_0$ in reverse mathematics.

We state two useful properties about indecomposable well-orderings. First, it is easy to show by induction that:

\begin{lem} \label{lem:indec_absorb}
If $M$ is indecomposable and $L_i$, $i<n$ each embed strictly into $M$, then $\left(\sum_{i<n} L_i\right)+M \equiv M$.
\end{lem}

Second, the following lemma will be useful for extracting information from embeddings between orderings.

\begin{lem} \label{lem:indec_embed_finite_sum}
Let $L$ be a linear ordering and let $M$ be an indecomposable well-ordering which does not embed into $L$. If $F$ embeds $M$ into a finite sum of $L$'s and $M$'s, then the range of $M$ under $F$ must be cofinal in some copy of $M$.

Therefore, if $M \cdot k$ embeds into a finite sum of $L$'s and $M$'s, then there must be at least $k$ many $M$'s in the sum.
\end{lem}
\begin{proof}
There are three cases regarding the position of the range of $M$ in the sum. \underline{Case 1.} $F$ maps some final segment of $M$ into some copy of $L$. Since $M$ is indecomposable, it follows that $M$ embeds into $L$, contradiction. \underline{Case 2.} $F$ maps some final segment of $M$ into a bounded segment of some copy of $M$. Since $M$ is indecomposable, that implies that $M$ maps into a bounded segment of itself. This contradicts well-foundedness of $M$. \underline{Case 3.} The remaining case is that the range of $M$ is cofinal in some copy of $M$, as desired.
\end{proof}

We remark that for our purposes, we do not need to pay attention to the computational content of the previous two lemmas. In addition, unlike in reverse mathematics, we do not need to distinguish between ``$M$ does not embed into $L$'' and ``$L$ strictly embeds into $M$''.

\section{An analog of Chen's theorem}

In this section, given a labeled well-ordering $\L = (L,0_L,S,p)$, $\langle Y_a \rangle_{a \in L}$ denotes the unique hierarchy on $L$, as defined in Proposition \ref{prop:ATR_labeled}. (This notation persists for the next two sections, which use results from this section.)

We present the technical ingredients needed for our reductions from $\ATR$ to theorems about embeddings between well-orderings. The main result is an analog of the following theorem of Chen, which suggests a bridge from computing jump hierarchies to comparing well-orderings. We will not need Chen's theorem so we will not define the notation therein; see Shore \cite[Theorem 3.5]{shore93} for details.

\begin{thm}[Chen {\cite[Corollary 10.2]{chen78}}] \label{thm:original_chen}
Fix $x \in \O$. There is a recursive function $k(a,n)$ such that for all $a <_\O x$ and $n \in \N$,
\begin{enumerate}
	\item $k(a,n)$ is an index for a recursive well-ordering $K(a,n)$;
	\item if $n \in H_a$, then $K(a,n)+1 \leq \omega^{|x|}$;
	\item if $n \notin H_a$, then $K(a,n) \equiv \omega^{|x|}$.
\end{enumerate}
\end{thm}

We adapt Chen's theorem to our setting, which involves well-orderings instead of notations. Our proof is a direct adaptation of Shore's proof of Chen's theorem. We  begin by defining some computable operations on trees.

\begin{defn}[Shore {\cite[Definition 3.9]{shore93}}, slightly modified]
For any (possibly finite) sequence of trees $\langle T_i \rangle$, we define their \emph{maximum} by joining all $T_i$'s at the root, i.e.,
\[ \max(\langle T_i \rangle) = \{\langle \rangle\} \cup \{i\concat\sigma: \sigma \in T_i\}. \]
Next, we define the \emph{minimum} of a sequence of trees to be their ``staggered common descent tree''. More precisely, for any (possibly finite) sequence of trees $\langle T_i \rangle$, a node at level $n$ of the tree $\min(\langle T_i \rangle)$ consists of, for each $i < n$ such that $T_i$ is defined, a chain in $T_i$ of length $n$. A node extends another node if for each $i$ in their common domain, the $i^{\text{th}}$ chain in the former node is an end-extension of the $i^{\text{th}}$ chain in the latter node.
\end{defn}

It is easy to see that the maximum and minimum operations play well with the ranks of trees:

\begin{lem}[Shore {\cite[Lemma 3.10]{shore93}}] \label{lem:max_min_properties}
Let $\langle T_i\rangle$ be a (possibly finite) sequence of trees.
\begin{enumerate}
	\item If $\rk(T_i) \leq \alpha$ for all $i$, then $\rk(\max(\langle T_i \rangle)) \leq \alpha$.
	\item If there is some $i$ such that $T_i$ is ill-founded, then $\max(\langle T_i \rangle)$ is ill-founded.
	\item If every $T_i$ is well-founded, then $\rk(\min(\langle T_i \rangle)) \leq \rk(T_i)+i$.
	\item If every $T_i$ is ill-founded, then $\min(\langle T_i \rangle)$ is ill-founded as well.
\end{enumerate}
\end{lem}

With the maximum and minimum operations in hand, we may prove an analog of Theorem 3.11 in Shore \cite{shore93}:

\begin{thm} \label{thm:g_h_trees}
Given a labeled well-ordering $\L$, we can uniformly compute sequences of trees $\langle g(a,n) \rangle_{n \in \N,a \in L}$ and $\langle h(a,n) \rangle_{n \in \N,a \in L}$ such that:
\begin{itemize}
	\item if $n \in Y_a$, then $\rk(g(a,n)) \leq \omega \cdot \otp(L\restriction a)$ and $h(a,n)$ is ill-founded;
	\item if $n \notin Y_a$, then $\rk(h(a,n)) \leq \omega \cdot \otp(L\restriction a)$ and $g(a,n)$ is ill-founded.
\end{itemize}
\end{thm}
\begin{proof}
We define $g$ and $h$ by $\L$-effective transfinite recursion on $L$. For the base case (recall $Y_{0_L}	 = \L$), define $g(0_L,n)$ to be an infinite path of $0$'s for all $n \notin \L$, and the empty node for all $n \in \L$. Define $h(0_L,n)$ analogously.

For $b$ limit, define $g(b,\langle a,n\rangle) = g(a,n)$ and $h(b,\langle a,n\rangle) = h(a,n)$ for any $n \in \N$ and $a <_L b$.

For $b = a+1$, fix a Turing functional $W$ which computes $X$ from $X'$ for any $X$. In particular,
\[ n \in Y_b \quad \text{iff} \quad (\exists \langle P,Q,n \rangle \in W)(P \subseteq Y_a \text{ and }Q \subseteq Y^c_a). \]
Then define
\[ h(b,n) = \max(\langle \min(\langle \{h(a,p): p \in P\},\{g(a,q): q \in Q\}\rangle ): \langle P,Q,n \rangle \in W\rangle ). \]

If $n \in Y_b$, then there is some $\langle P,Q,n \rangle \in W$ such that $P \subseteq Y_a$ and $Q \subseteq Y_a^c$. Then every tree in the above minimum for $\langle P,Q,n \rangle$ is ill-founded, so the minimum is itself ill-founded. Hence $h(b,n)$ is ill-founded.

If $n \notin Y_b$, then for all $\langle P,Q,n \rangle \in W$, either $P \not\subseteq Y_a$ or $Q \not\subseteq Y^c_a$. Either way, all of the above minima have rank $< \omega\cdot \otp(L\restriction a)+\omega$. Hence $h(b,n)$ has rank at most $\omega \cdot \otp(L\restriction a)+\omega \leq \omega \cdot \otp(L\restriction b)$.

Similarly, define
\[ g(b,n) = \min(\langle \max(\langle \{g(a,p): p \in P\},\{h(a,q): q \in Q\}\rangle ): \langle P,Q,n \rangle \in W\rangle ). \]

This completes the construction for the successor case.
\end{proof}

Next, we adapt the above construction to obtain well-founded trees. To that end, for each well-ordering $L$, we aim to compute a tree $(T(\omega \cdot L))^\infty$ which is universal for all trees of rank $\leq \omega\cdot \otp(L)$. Shore \cite[Definition 3.12]{shore93} constructs such a tree by effective transfinite recursion. Instead, we use a simpler construction of Greenberg and Montalb\'an \cite{gm08}.

\begin{defn} \label{defn:tree_decreasing_seq}
Given a linear ordering $L$, define $T(L)$ to be the tree of finite $<_L$-decreasing sequences, ordered by extension.	
\end{defn}

It is easy to see that $L$ is well-founded if and only if $T(L)$ is well-founded, and if $L$ is well-founded, then $\rk(T(L)) = \otp(L)$.

\begin{defn}[{\cite[Definition 3.20]{gm08}}] \label{defn:fat_tree}
Given a tree $T$, define a tree
\[ T^\infty = \{\langle (\sigma_0,n_0),\dots,(\sigma_k,n_k)\rangle: \langle \rangle \neq \sigma_0 \subsetneq \dots \subsetneq \sigma_k \in T, n_0,\dots,n_k \in \N\}, \]
ordered by extension.
\end{defn}

\begin{lem}[{\cite[$\S$3.2.2]{gm08}}] \label{lem:fat_tree_properties}
Let $T$ be well-founded. Then
\begin{enumerate}
	\item $T^\infty$ is well-founded and $\rk(T^\infty) = \rk(T)$;
	\item for every $\sigma \in T^\infty$ and $\gamma < \rk_{T^\infty}(\sigma)$, there are infinitely many immediate successors $\tau$ of $\sigma$ in $T^\infty$ such that $\rk_{T^\infty}(\tau) = \gamma$;
	\item $\KB(T)$ embeds into $\KB(T^\infty)$;
	\item $\KB(T^\infty) \equiv \omega^{\rk(T)}+1$, hence $\KB(T^\infty)-\{\emptyset\}$ is indecomposable.
	\item if $S$ is well-founded and $\rk(S) \leq \rk(T)$ ($\rk(S) < \rk(T)$ resp.), then $\KB(S)$ embeds (strictly resp.) into $\KB(T^\infty)$.
\end{enumerate}
\end{lem}
\begin{proof}
(3) and (5) are not stated in \cite{gm08}, so we give a proof. By (1), fix a rank function $r: T \to \rk(T^\infty)+1$. We construct an embedding $f: T \to T^\infty$ which preserves rank (i.e., $r(\sigma) = \rk_{T^\infty}(f(\sigma))$), $<_\KB$, and level. Start by defining $f(\emptyset) = \emptyset$. Note that $r(\emptyset) = \rk(T^\infty) = \rk_{T^\infty}(\emptyset)$.

Suppose we have defined $f$ on $\sigma \in T$. Then, we extend $f$ by mapping each immediate successor $\tau$ of $\sigma$ to an immediate successor $f(\tau)$ of $f(\sigma)$ such that $r(\tau) = \rk_{T^\infty}(f(\tau))$. Such $f(\tau)$ exists by (2). Furthermore, by (2), if we start defining $f$ from the leftmost immediate successor of $\sigma$ and proceed to the right, we can extend $f$ in a way that preserves $<_\KB$. This proves (3).

(5) follows from (3) applied to $S$ and (4) applied to $S$ and $T$.
\end{proof}

Finally, we prove our analog of Chen's theorem (Theorem \ref{thm:original_chen}):

\begin{thm} \label{thm:analogue_chen}
Given a labeled well-ordering $\L$, we can uniformly compute an indecomposable well-ordering $M$ and well-orderings $\langle K(a,n) \rangle_{n \in \N,a \in L}$ such that:
\begin{itemize}
	\item if $n \in Y_a$, then $K(a,n) \equiv M$.
	\item if $n \notin Y_a$, then $K(a,n) < M$.
\end{itemize}
\end{thm}
\begin{proof}
Given $\L$, we may use Theorem \ref{thm:g_h_trees}, Definition \ref{defn:tree_decreasing_seq} and Definition \ref{defn:fat_tree} to uniformly compute
\begin{align*}
M &= \KB(T(\omega \cdot L)^\infty)-\{\emptyset\} \\
K(a,n) &= \KB(\min\{T(\omega \cdot L)^\infty,h(a,n)\})-\{\emptyset\} & \text{for }n \in \N,a \in L.
\end{align*}
By Lemma \ref{lem:fat_tree_properties}(4), $M$ is indecomposable. Also,
\begin{align*}
\rk(T(\omega \cdot L)^\infty) &= \omega\cdot \otp(L) \\
\text{so}\qquad \rk(\min\{T(\omega \cdot L)^\infty,h(a,n)\}) &\leq \omega\cdot \otp(L).
\end{align*}
It then follows from Lemma \ref{lem:fat_tree_properties}(5) that $K(a,n) \leq M$.

If $n \in Y_a$, then $h(a,n)$ is ill-founded. Fix some descending sequence $\langle \sigma_i \rangle_i$ in $h(a,n)$. Then we may embed $T(\omega \cdot L)^\infty$ into \\ $\min\{T(\omega \cdot L)^\infty,h(a,n)\}$ while preserving $<_{\KB}$: map $\tau$ to $\langle\langle \tau\restriction i,\sigma_i \rangle\rangle_{i=0}^{|\tau|}$. Therefore $M \leq K(a,n)$, showing that $K(a,n) \equiv M$ in this case.

If $n \notin Y_a$, then $\rk(h(a,n)) \leq \omega\cdot \otp(L\restriction a)$. Therefore
\[ \rk(\min\{T(\omega\cdot L)^\infty,h(a,n)\}) \leq \omega\cdot\otp(L\restriction a)+1. \]
Since $\omega\cdot\otp(L\restriction a)+1 < \omega\cdot\otp(L)$, by Lemma \ref{lem:fat_tree_properties}(5), $K(a,n) < M$.
\end{proof}

\section{Reducing $\ATR$ to $\WCWO$}

In this section, we apply Theorem \ref{thm:analogue_chen} to show that $\ATR \leq_W \WCWO$ (Theorem \ref{thm:ATR_leq_WCWO}). Together with Proposition \ref{prop:WQO_WO_ATR_CWO_ATR}, that implies that $\WCWO \equiv_W \CWO \equiv_W \ATR$.

First we work towards some sort of modulus for jump hierarchies. The next two results are adapted from Shore \cite[Theorem 2.3]{shore93}. We have added uniformities where we need them.

\begin{prop} \label{prop:Y_a_Pi01_singleton}
Given a labeled well-ordering $\L$ and $a \in L$, we can uniformly compute an index for a $\Pi^{0,\L}_1$-singleton $\{f\}$ which is strictly increasing, and Turing reductions witnessing that $f \equiv_T Y_a$.
\end{prop}
\begin{proof}
By $\L$-effective transfinite recursion on $L$, we can compute an index for $Y_a$ as a $\Pi^{0,\L}_2$-singleton (see Sacks \cite[Proposition II.4.1]{sacks}). Define $f$ to be the join of $Y_a$ and the lex-minimal Skolem function which witnesses that $Y_a$ satisfies the $\Pi^{0,\L}_2$ predicate that we computed. Then we can compute an index for $f$ as a $\Pi^{0,\L}_1$-singleton (see Jockusch, McLaughlin \cite[Theorem 3.1]{jm69}). Clearly we can compute Turing reductions witnessing that $Y_a \leq_T f \leq_T \L \oplus Y_a$. Next, we can $\L$-uniformly compute a Turing reduction from $Y_{0_L} = \L$ to $Y_a$, and hence a Turing reduction from $\L \oplus Y_a$ to $Y_a$.

Finally, without loss of generality, we can replace $f: \N \to \N$ with its cumulative sum, which is strictly increasing.
\end{proof}

\begin{lem} \label{lem:uniformly_majorreducible}
There are indices $e_0$, $e_1$, and $e_2$ such that for all labeled well-orderings $\L$ and $a \in L$, there is some strictly increasing $f: \N \to \N$ such that if $Y_a$ is the $a^{\text{th}}$ column of the unique hierarchy on $L$, then:
\begin{enumerate}
	\item $\Phi_{e_0}^{\L \oplus a}$ is an index for a Turing reduction from $f$ to $Y_a$;
	\item for all $g: \N \to \N$, $\Phi^{\L \oplus a \oplus g}_{e_1}(0)\conv$ if and only if $g$ does not majorize $f$;
	\item for all $g$ which majorizes $f$, $\Phi^{\L \oplus a \oplus g}_{e_2}$ is total and defines $Y_a$.
\end{enumerate}
\end{lem}
\begin{proof}
Given $\L$ and $a \in L$, first use Proposition \ref{prop:Y_a_Pi01_singleton} to compute a tree $T$ with a unique path $f$ which is strictly increasing, and Turing reductions witnessing that $f \equiv_T Y_a$. This shows (1).

Given $g: \N \to \N$, we can compute the $g$-bounded subtree $T_g$ of $T$. If $g$ does not majorize $f$, then $T_g$ has no infinite path. In that case, $T_g$ is finite by K\"onig's lemma, hence we can eventually enumerate that fact. This shows (2).

If $g$ majorizes $f$, then we can compute $f$ as follows: $\sigma \prec f$ if and only if for all other $\tau$ with $|\tau| = |\sigma|$, the $g$-bounded subtree of $T$ above $\tau$ is finite. We can then compute $Y_a$ from $f$. This shows (3).
\end{proof}

We now combine Theorem \ref{thm:analogue_chen} with the above lemma to prove that

\begin{thm} \label{thm:ATR_leq_WCWO}
$\ATR \leq_W \WCWO$.
\end{thm}
\begin{proof}
We reduce the version of $\ATR$ in Proposition \ref{prop:ATR_labeled} to $\WCWO$. Given a labeled well-ordering $\L$ and $a \in L$, by Lemma \ref{lem:uniformly_majorreducible}, there is some strictly increasing $f$ such that if $g$ majorizes $f$, then $\L \oplus a \oplus g$ uniformly computes $Y_a$.

Furthermore, we may compute reductions witnessing $\range(f) \leq_T f \leq_T Y_a$. From that we may compute a many-one reduction $r$ from $\range(f)$ to $Y_{a+1}$ (the $(a+1)^{\text{th}}$ column of the unique hierarchy on $(L\restriction \{b: b \leq_L a\})+1$).

Next, use $\L$ to compute labels for $(L\restriction \{b: b \leq_L a\})+1$. Apply Theorem \ref{thm:analogue_chen} to $(L\restriction \{b: b \leq_L a\})+1$ (and its labels) to compute an indecomposable well-ordering $M$ and for each $n$, a well-ordering $L_n := K(a+1,r(n))$, such that
\begin{align*}
n \in \range(f) \quad &\Leftrightarrow \quad r(n) \in Y_{a+1} \quad \Leftrightarrow \quad L_n \equiv M \\
n \notin \range(f) \quad &\Leftrightarrow \quad r(n) \notin Y_{a+1} \quad \Leftrightarrow \quad L_n < M.
\end{align*}

For the forward functional, consider the following $\WCWO$-instance:
\[ \sum_n M \quad \text{and} \quad \left(\sum_n L_n\right)+1. \]

Observe that by Lemma \ref{lem:indec_absorb}, $\sum_n L_n$ has the same ordertype as $\sum_n M$. Hence any $\WCWO$-solution $F$ must go from left to right. Furthermore, since $M$ is indecomposable, it has no last element, so $F$ must embed $\sum_n M$ into $\sum_n L_n$.

For the backward functional, we start by uniformly computing any element $m_0$ of $M$. Then we use $F$ to compute the following function:
\[ g(n) = \pi_0(F(\langle n+1,m_0 \rangle)). \]

We show that $g$ majorizes $f$. For each $n$, $F$ embeds $M \cdot n$ into $\sum_{i \leq g(n)} L_i$. It follows from Lemma \ref{lem:indec_embed_finite_sum} that at least $n$ of the $L_i$'s ($i \leq g(n)$) must have ordertype $M$. That means that there must be at least $n$ elements in the range of $f$ which lie below $g(n)$, i.e., $f(n) \leq g(n)$.

Since $g$ majorizes $f$, $\L \oplus a \oplus g$ uniformly computes $Y_a$ by Lemma \ref{lem:uniformly_majorreducible}, as desired.
\end{proof}

Using Theorem \ref{thm:ATR_leq_WCWO} and Proposition \ref{prop:WQO_WO_ATR_CWO_ATR}, we conclude that 

\begin{cor} \label{cor:CWO_ATR_WCWO}
$\CWO \equiv_W \ATR \equiv_W \WCWO$.
\end{cor}

\section{Reducing $\ATR$ to $\NDS_{\WO}$ and $\NIAC_{\WO}$}

Shore \cite[Theorem 3.7]{shore93} showed that in reverse mathematics, $\NDS_{\WO}$ (formulated as a $\Pi^1_2$ sentence) implies $\ATR_0$ over $\RCA_0$. We adapt his proof to show that

\begin{thm} \label{thm:ATR_leq_c_NDS_WO}
$\ATR \leq_W \C_\N \ast \NDS_\WO$. In particular, $\ATR \leq_c \NDS_{\WO}$ and $\ATR \leq_W^{\arith} \NDS_\WO$.
\end{thm}
\begin{proof}
We reduce the version of $\ATR$ in Proposition \ref{prop:ATR_labeled} to $\NDS_\WO$. Given a labeled well-ordering $\L$ and $a \in L$, by Lemma \ref{lem:uniformly_majorreducible}, there is some strictly increasing $f$ such that if $g$ majorizes $f$, then $\L \oplus a \oplus g$ uniformly computes $Y_a$. Furthermore, as in the proof of Theorem \ref{thm:ATR_leq_WCWO}, we may compute a many-one reduction $r$ from $f$ to $Y_{a+1}$.

Next, use $\L$ to compute labels for $(L\restriction \{b: b \leq_L a\})+1$. Apply Theorem \ref{thm:analogue_chen} to $(L\restriction \{b: b \leq_L a\})+1$ to compute an indecomposable well-ordering $M$ and for each $i$ and $n$, a well-ordering $K(a+1,r(i,n))$, such that
\begin{align*}
f(i) = n \quad &\Leftrightarrow \quad r(i,n) \in Y_{a+1} \quad \Leftrightarrow \quad K(a+1,r(i,n)) \equiv M \\
f(i) \neq n \quad &\Leftrightarrow \quad r(i,n) \notin Y_{a+1} \quad \Leftrightarrow \quad K(a+1,r(i,n)) < M.
\end{align*}

For the forward functional, define for each $j$ and $n$:
\begin{align*}
L_{j,n} &= \sum_{j \leq i < n} K(a+1,r(i,n)) \\
N_j &= \sum_n L_{j,n}.
\end{align*}
For each $j$ and $n$, $L_{j+1,n}$ uniformly embeds into $L_{j,n}$. So for each $j$, we can uniformly embed $N_{j+1}$ into $N_j$. Hence $\langle N_j \rangle_j$ (with said embeddings) is an $\NDS_{\WO}$-instance.

Apply $\NDS_{\WO}$ to obtain some embedding $F: N_j \to N_k$, $j<k$. For the backward functional, we aim to compute a sequence $\langle h_q \rangle_q$ of functions, such that $h_q$ majorizes $f$ for all sufficiently large $q$. We start by uniformly computing any element $m_0$ of $M$. Then for each $q$, define
\[ h_q(0) = q \quad \text{and} \quad h_q(n+1) = \pi_0(F(\langle h_q(n)+1,m_0 \rangle)). \]

We show that $h_{f(k)}$ majorizes $f$. (Hence for all $q \geq f(k)$, $h_q$ majorizes $f$.) For this proof, temporarily set $q = f(k)$. We show by induction on $n$ that $h_q(n) \geq f(k+n)$. The base case $n=0$ holds by definition of $q$. 

Suppose $h_q(n) \geq f(k+n)$. For each $j \leq i \leq k+n$, $K(a+1,r(i,f(i)))$ is a summand in $L_{j,f(i)}$ (because $f(i) > i$), which is in turn a summand in $\sum_{m \leq h_q(n)} L_{j,m}$. That implies that $M \cdot (k+n-j+1)$ embeds into $\sum_{m \leq h_q(n)} L_{j,m}$, which lies below $\langle h_q(n)+1,m_0 \rangle$ in $N_j$.

Composing with $F$, we deduce that $M \cdot (k+n-j+1)$ embeds into the initial segment of $N_k$ below $F(\langle h_q(n)+1,m_0\rangle)$, which is contained in $\sum_{m \leq h_q(n+1)} L_{k,m}$. It follows from Lemma \ref{lem:indec_embed_finite_sum} that there are at least $k+n-j+1$ many copies of $M$ in $\sum_{m \leq h_q(n+1)} L_{k,m}$. Therefore, there are at least $k+n-j+1$ many elements in $\{f(i): i \geq k\}$ below $h_q(n+1)$. It follows that
\[ h_q(n+1) \geq f(k+n-j+k) \geq f(k+n+1) \]
as desired. This completes the proof of the inductive step. We have shown that $h_{f(k)}$ majorizes $f$.

Finally, by Lemma \ref{lem:uniformly_majorreducible}(2), given $\L \oplus a \oplus \langle h_q \rangle_q$, we may apply $\C_\N$ (Definition \ref{defn:LPO_and_choice_problems}) to compute some $q$ such that $h_q$ majorizes $f$. Then $\L \oplus a \oplus h_q$ uniformly computes $Y_a$ by Lemma \ref{lem:uniformly_majorreducible}(3), as desired.
\end{proof}

The above proof can be easily modified to show that

\begin{thm}
$\ATR \leq_W \C_\N \ast \NIAC_\WO$. In particular, $\ATR \leq_c \NIAC_{\WO}$ and $\ATR \leq_W^{\arith} \NIAC_\WO$.
\end{thm}
\begin{proof}
Given $\L$ and $a \in L$, compute $\langle L_{j,n} \rangle_{j,n}$ and $\langle N_j \rangle_j$ as in the proof of Theorem \ref{thm:ATR_leq_c_NDS_WO}. Then consider the $\NIAC_{\WO}$-instance $\langle N_j+j \rangle_j$.

Given an embedding $F: N_j+j \to N_k+k$, first observe that by Lemma \ref{lem:indec_absorb}, $N_j$ and $N_k$ have the same ordertype, namely that of $M \cdot \omega$. Hence $j < k$. Furthermore, since $M$ is indecomposable, $F$ must embed $N_j$ into $N_k$. The backward functional is then identical to that in Theorem \ref{thm:ATR_leq_c_NDS_WO}.
\end{proof}

We do not know if $\ATR \leq_W \NDS_{\WO}$, $\ATR \leq_W \NIAC_{\WO}$, or even $\ATR \leq_W \WQO_{\WO}$.

\section{Two-sided problems} \label{section:two_sided}

Many of the problems we have considered thus far have domains which are $\Pi^1_1$. For instance, the domain of $\CWO$ is the set of pairs of well-orderings. In that case, being outside the domain is a $\Sigma^1_1$ property. Now, any $\Sigma^1_1$ property can be thought of as a problem whose instances are sets satisfying said property and solutions are sets which witness that said property holds. This suggests that we combine a problem which has a $\Pi^1_1$ domain with the problem corresponding to the complement of its domain.

One obvious way to combine such problems is to take their union. For example, a ``two-sided'' version of $\WCWO$ could map pairs of well-orderings to any embedding between them, and map other pairs of linear orderings to any infinite descending sequence in either linear ordering. We will not consider such problems in this paper, because they are not Weihrauch reducible (or even arithmetically Weihrauch reducible) to $\C_{\N^\N}$. (Any such reduction could be used to give a $\Sigma^1_1$ definition for the set of indices of pairs of well-orderings. See also Brattka, de Brecht, Pauly \cite[Theorem 7.7]{bdp12}.) On the other hand, it is not hard to see that the problems corresponding to Fra\"iss\'e's conjecture ($\WQO_\LO$) and K\"onig's duality theorem (see section \ref{section:KDT}) are Weihrauch reducible to $\C_{\N^\N}$.

However, note that embeddings between linear orderings can still exist even if either linear ordering is ill-founded! This suggests an alternative method of combination, resulting in the following ``two-sided'' extensions of $\CWO$ and $\WCWO$.

\begin{defn}
Define the following problems:
\begin{itemize}
	\item[$\CWO_2$:] Given linear orderings $L$ and $M$, either produce an embedding from one of them onto an initial segment of the other, or an infinite descending sequence in either ordering. In either case we indicate which type of solution we produce.
	\item[$\WCWO_2$:] Given linear orderings $L$ and $M$, either produce an embedding from one of them into the other, or an infinite descending sequence in either ordering. In either case we indicate which type of solution we produce.
\end{itemize}
\end{defn}

It is not hard to see that whether solutions to instances of the above problems come with an indication of their type does not affect the Weihrauch degree of the problems. Hence we include the type for our convenience.

Next, we define a two-sided version of $\ATR$. In section \ref{section:KDT}, we will show that it is closely related to K\"onig's duality theorem (Theorem \ref{thm:ATR_2_leq_c_KDT}).

Recall our definition of a jump hierarchy:

\setcounter{sectiontmp2}{\value{section}}
\setcounter{thmtmp2}{\value{thm}}
\setcounter{section}{\value{sectiontmp1}}
\setcounter{thm}{\value{thmtmp1}}

\begin{defn}
Given a linear ordering $L$ with first element $0_L$ and a set $A \subseteq \N$, a \emph{jump hierarchy on $L$ which begins with $A$} is a set $\langle X_a\rangle_{a \in L}$ such that
\begin{itemize}
	\item $X_{0_L} = A$;
	\item for every $b \in L$, $X_b = \left(\bigoplus_{a <_L b} X_a\right)'$.
\end{itemize}
\end{defn}

\setcounter{section}{\value{sectiontmp2}}
\setcounter{thm}{\value{thmtmp2}}

Jump hierarchies on ill-founded linear orderings were first studied by Harrison \cite{harrison68}, and are often called pseudohierarchies. See, for example, \cite[Section V.4]{sim_book}).

\begin{defn} \label{defn:two_sided_ATR}
We define a two-sided version of $\ATR$ as follows:
\begin{itemize}
	\item[$\ATR_2$:] Given a linear ordering $L$ and a set $A \subseteq \N$, either produce an infinite $<_L$-descending sequence $S$, or a jump hierarchy $\langle X_a\rangle_{a \in L}$ on $L$ which begins with $A$. In either case we indicate which type of solution we produce.\footnote{Just as for $\CWO_2$ and $\WCWO_2$, this does not affect the Weihrauch degree of $\ATR_2$.} 
\end{itemize}
\end{defn}

Just as for $\CWO$ and $\WCWO$, if we require an $\ATR_2$-solution to an ill-founded $L$ to be an infinite $<_L$-descending sequence, then the resulting problem is not Weihrauch reducible to $\C_{\N^\N}$. The same holds if we require an $\ATR_2$-solution to $L$ to be a jump hierarchy whenever $L$ supports a jump hierarchy, because

\begin{thm}[Harrington, personal communication]
The set of indices for linear orderings which support a jump hierarchy is $\Sigma^1_1$-complete.
\end{thm}

A Weihrauch reduction from the aforementioned variant of $\ATR_2$ to $\C_{\N^\N}$ would yield a $\Pi^1_1$ definition of the set of indices for linear orderings which support a jump hierarchy, contradicting Harrington's result.

Next, we determine the positions of $\CWO_2$, $\WCWO_2$, and $\ATR_2$ relative to $\UC_{\N^\N}$ and $\C_{\N^\N}$ in the Weihrauch degrees. In addition, even though we are not viewing $\WQO_\LO$ (Fra\"iss\'e's conjecture) as a two-sided problem, most of our arguments and results hold for $\WQO_\LO$ as well.

First observe that each of $\CWO$, $\WCWO$, and $\ATR$ is trivially Weihrauch reducible to its two-sided version. By Corollary \ref{cor:CWO_ATR_WCWO} and the fact that $\ATR \equiv_W \UC_{\N^\N}$ (Kihara, Marcone, Pauly \cite{kmp18}), these two-sided problems lie above $\UC_{\N^\N}$ in the Weihrauch degrees. We do not know if $\WQO_\LO$ lies above $\UC_{\N^\N}$ in the Weihrauch degrees.

Next observe that $\CWO_2$, $\WCWO_2$, $\ATR_2$, and $\WQO_\LO$ are each defined by an arithmetic predicate on an arithmetic domain. It easily follows that they lie below $\C_{\N^\N}$ in the Weihrauch degrees. In fact, they lie strictly below $\C_{\N^\N}$:

\begin{prop} \label{prop:domain_not_Sigma11_no_arith_reduction}
Suppose that $P$ is an arithmetically defined multivalued function such that $\dom(P)$ is not $\Pi^1_1$. If $Q$ is arithmetically defined and $\dom(Q)$ is arithmetic, then $P$ is not arithmetically Weihrauch reducible to $Q$.
\end{prop}
\begin{proof}
If $P$ is arithmetically Weihrauch reducible to $Q$ via arithmetically defined functionals $\Phi$ and $\Psi$, then we could give a $\Pi^1_1$ definition for $\dom(P)$ as follows: $X \in \dom(P)$ if and only if
\[ \Phi(X) \in \dom(Q) \land \forall Y[Y \in Q(\Phi(X)) \to \Psi(X \oplus Y) \in P(X)]. \]
Contradiction.
\end{proof}

\begin{cor} \label{cor:C_N_N_not_below_two_sided_problems}
$\C_{\N^\N}$ is not arithmetically Weihrauch reducible to any of $\CWO_2$, $\WCWO_2$, $\ATR_2$, or $\WQO_\LO$. 
\end{cor}
\begin{proof}
Each of $\CWO_2$, $\WCWO_2$, $\ATR_2$, and $\WQO_\LO$ are arithmetically defined with arithmetic domain. $\C_{\N^\N}$ is also arithmetically defined, but its domain is $\Sigma^1_1$-complete. Apply Proposition \ref{prop:domain_not_Sigma11_no_arith_reduction}.
\end{proof}

Next we show that $\CWO_2$, $\WCWO_2$, $\ATR_2$, and $\WQO_{\LO}$ are not Weihrauch reducible (or even computably reducible) to $\UC_{\N^\N}$. First we have a boundedness argument:

\begin{lem} \label{lem:hyp_solution_bound}
Suppose $P(X,Y)$ is a $\Pi^1_1$ predicate and $D$ is a $\Sigma^1_1$ set of reals. If for every $X \in D$, there is some hyperarithmetic $Y$ such that $P(X,Y)$ holds, then there is some $b \in \O$ such that for every $X \in D$, there is some $Y \leq_T H_b$ such that $P(X,Y)$.
\end{lem}
\begin{proof}
Consider the following $\Pi^1_1$ predicate of $X$ and $a$:
\[ X \notin D \lor (a \in \O \land (\exists e)(\Phi^{H_a}_e \text{ is total and }P(X,\Phi^{H_a}_e))). \]
By $\Pi^1_1$-uniformization, there is a $\Pi^1_1$ predicate $Q(X,a)$ uniformizing it. Then the set
\[ \{a: (\exists X \in D)Q(X,a)\} = \{a: (\exists X \in D)(\forall b \neq a)\neg Q(X,b)\} \]
is $\Sigma^1_1$ and contained in $\O$. Therefore it is bounded by some $b \in \O$, proving the desired statement.
\end{proof}

\begin{cor} \label{cor:WCWO_CWO_ATR2_WQO_LO_rec_instance_no_hyp_solution}
Each of $\WCWO_2$, $\CWO_2$, $\ATR_2$, and $\WQO_{\LO}$ have a computable instance with no hyperarithmetic solution. 
\end{cor}
\begin{proof}
By the contrapositive of Lemma \ref{lem:hyp_solution_bound}, it suffices to show that for all $b \in \O$, there is a computable instance of each problem with no $H_b$-computable solution.

Observe that for all $b \in \O$, there is a computable instance of $\ATR$ such that none of its solutions are computable in $H_b$.\footnote{Note that the domain of $\ATR$ is not $\Sigma^1_1$, so we cannot apply Lemma \ref{lem:hyp_solution_bound} to show that there is a computable instance of $\ATR$ with no hyperarithmetic solution. (The latter statement is clearly false.)} The following reductions imply that the same holds for $\WCWO_2$, $\CWO_2$, $\ATR_2$, and $\WQO_\LO$:
\begin{align*}
\ATR &\leq_W \WCWO \leq_W \WCWO_2 \leq_W \CWO_2 & \text{Theorem \ref{thm:ATR_leq_WCWO}} \\
\ATR &\leq_W \ATR_2 \\
\ATR &\leq_c \WQO_\LO &\text{Theorem \ref{thm:ATR_leq_c_NDS_WO}}
\end{align*}
This completes the proof.
\end{proof}

Corollary \ref{cor:WCWO_CWO_ATR2_WQO_LO_rec_instance_no_hyp_solution} implies that

\begin{cor} \label{cor:two_sided_not_below_UC}
$\WCWO_2$, $\CWO_2$, $\ATR_2$, and $\WQO_{\LO}$ are not computably reducible or arithmetically Weihrauch reducible to $\UC_{\N^\N}$.
\end{cor}

\subsection{$\ATR_2$ and variants thereof}

In this subsection, we prove some results regarding $\ATR_2$ and its variants. First we have several results showing that $\ATR_2$ is fairly robust. Next we show that $\CWO_2 \leq_W \ATR_2$ (Theorem \ref{thm:CWO_2_leq_W_ATR_2}), in analogy with $\CWO \leq_W \ATR$ (Proposition \ref{prop:WQO_WO_ATR_CWO_ATR}).

We start with the following analog of Proposition \ref{prop:ATR_arith_formula}:

\begin{prop} \label{prop:ATR2_arith_formula}
$\ATR_2$ is Weihrauch equivalent to the following problem. Instances are triples $(L,A,\Theta)$ where $L$ is a linear ordering, $A \subseteq \N$, and $\Theta(n,Y,A)$ is an arithmetical formula whose only free variables are $n$, $Y$ and $A$. Solutions are either infinite $<_L$-descending sequences, or hierarchies $\langle Y_a \rangle_{a \in L}$ such that for all $b \in L$, $Y_b = \{n: \Theta(n,\bigoplus_{a <_L b} Y_a,A)\}$. (As usual, solutions come with an indication of their type.)
\end{prop}
\begin{proof}
Roughly speaking, we extend the reductions defined in Proposition \ref{prop:ATR_arith_formula}. First, $\ATR_2$ is Weihrauch reducible to the above problem: for the forward reduction, given $(L,A)$, consider $(L,A,\Theta)$ where $\Theta(n,Y,A)$ holds if either $Y = \emptyset$ and $n \in A$, or $n \in Y'$. The backward reduction is the identity.

Conversely, given $(L,A,\Theta)$, let $k$ be one greater than the number of quantifier alternations in $\Theta$. Apply $\ATR_2$ to $(1 + k \cdot L+2,L \oplus A)$. If we obtain an infinite descending sequence in $1+k\cdot L+2$, we can uniformly compute an infinite descending sequence in $L$ and output that.

Otherwise, we obtain a jump hierarchy $\langle X_\alpha \rangle_{\alpha \in 1+k\cdot L+2}$. We want to use it to either compute a hierarchy on $L$, or an infinite $<_L$-descending sequence.

We start by using the recursion theorem to compute a $\langle X_{(a,k-1)} \rangle_{a \in L}$-partial recursive function $f: L \to \N$, as described in the proof of Proposition \ref{prop:ATR_arith_formula}. Note that $f$ may not be total.

Next, we compute $(\langle X_{(a,k-1)} \rangle_{a \in L})''$ and use that to decide whether $f$ is total. If so, following the proof of Proposition \ref{prop:ATR_arith_formula}, we may compute a hierarchy on $L$ with the desired properties.

If not, we use $(\langle X_{(a,k-1)} \rangle_{a \in L})''$ to compute the complement of the domain of $f$ in $L$. This set has no $<_L$-least element, by construction of $f$. Therefore, we can uniformly compute an infinite $<_L$-descending sequence within it.
\end{proof}

Just as we defined labeled well-orderings, we may also define labeled linear orderings if said linear orderings have first elements. Then we have the following analog of Proposition \ref{prop:ATR_equiv_ATR_labels}:

\begin{prop} \label{prop:ATR_2_equiv_ATR_2_labels}
$\ATR_2$ is Weihrauch equivalent to the following problem: an instance is a labeled linear ordering $\L$ and a set $A \subseteq \N$, and a solution is an $\ATR_2$-solution to $(L,A)$.
\end{prop}
\begin{proof}
It suffices to reduce $\ATR_2$ to the given problem. Given $(L,A)$, we start by computing $\omega \cdot (1+L)$ and labels for it. Then we apply the given problem to $\omega \cdot (1+L)$ (and its labels) and the set $L \oplus A$.

If we obtain an infinite descending sequence in $\omega\cdot(1+L)$, we can uniformly compute an infinite descending sequence in $L$ and output that.

Otherwise, we obtain a jump hierarchy $\langle X_{(n,\alpha)} \rangle_{n \in \omega,\alpha \in 1+L}$ which starts with $L \oplus A$. First use this hierarchy to compute $L''$, which tells us whether $L$ has a first element. If not, we can uniformly compute an infinite descending sequence in $L$ and output that.

Otherwise, we use the recursion theorem to compute a partial $\langle X_{(0,b)} \rangle_{b \in L}$-recursive function $f: L \to \N$, as described in the proof of Proposition \ref{prop:ATR_equiv_ATR_labels}. Then we compute
\[ S = \left\{b \in L: \langle \Phi^{X_{(0,a)}}_{f(a)} \rangle_{a <_L b} \text{ defines a jump hierarchy}\right\} \]
and consider two cases.

\underline{Case 1.} If $S$ is all of $L$, then we output $\langle \Phi^{X_{(0,a)}}_{f(a)} \rangle_{a \in L}$, which is a jump hierarchy on $L$ which starts with $A$.

\underline{Case 2.} Otherwise, observe that by construction of $f$, $L\backslash S$ has no $<_L$-least element. Then we can compute an infinite $<_L$-descending sequence in $L\backslash S$ and output that.

Finally, note that $\langle X_{(n,\alpha)} \rangle_{n \in \omega,\alpha \in 1+L}$ can compute the above case division and the output in each case.
\end{proof}

Proposition \ref{prop:ATR_2_equiv_ATR_2_labels} will be useful in section \ref{section:KDT}. Using similar ideas, we can show that

\begin{prop} \label{prop:ATR2_nonuniform_jump_hierarchy}
$\ATR_2$ is arithmetically Weihrauch equivalent to the following problem: an instance is a linear ordering $L$ and a set $A \subseteq \N$, and a solution is an infinite $<_L$-descending sequence, or some $\langle X_a \rangle_{a \in L}$ such that $X_{0_L} = A$ and $X'_a \leq_T X_b$ for all $0_L \leq_L a <_L b$.
\end{prop}

\begin{proof}
It suffices to construct an arithmetic Weihrauch reduction from $\ATR_2$ to the given problem. Given $(L,A)$, the forward functional outputs $(L,L \oplus A)$. To define the backward functional: if the above problem gives us some infinite $<_L$-descending sequence then we output that. Otherwise, suppose we are given $\langle X_a \rangle_{a \in L}$ such that $X_{0_L} = A$ and $X'_a \leq_T X_b$ for all $0_L \leq_L a <_L b$.

We start by attempting to use $(\langle X_a \rangle_{a \in L})'''$-effective transfinite recursion along $L$ to define a partial $(\langle X_a \rangle_{a \in L})'''$-recursive function $f: L \to \N$ such that $\langle \Phi^{X_a}_{f(a)} \rangle_{a \in L}$ is a jump hierarchy on $L$ which starts with $A$.

For the base case, we use $X_{0_L} = L \oplus A$ to uniformly compute $A$. For $b >_L 0_L$, first use $(\bigoplus_{a \leq_L b} X_a)'''$ to find Turing reductions (for each $a <_L b$) witnessing that $X'_a \leq_T X_b$. Then we can use $X_b$ to compute $(\bigoplus_{a <_L b} \Phi^{X_a}_{f(a)})'$. This completes the definition of $f$.

Next, compute
\[ S = \left\{b \in L: \langle \Phi^{X_a}_{f(a)} \rangle_{a <_L b} \text{ defines a jump hierarchy}\right\} \]
and consider two cases.

\underline{Case 1.} If $S$ is all of $L$, then we output $\langle \Phi^{X_a}_{f(a)} \rangle_{a \in L}$, which is a jump hierarchy on $L$ which starts with $A$.

\underline{Case 2.} Otherwise, observe that by construction of $f$, $L\backslash S$ has no $<_L$-least element. Then we can compute an infinite $<_L$-descending sequence in $L\backslash S$ and output that.

Finally, note that by choosing $n$ sufficiently large, $(\langle X_a \rangle_{a \in L})^{(n)}$ can compute the above case division and the output in each case.
\end{proof}

Next, in analogy with $\CWO \leq_W \ATR$ (Proposition \ref{prop:WQO_WO_ATR_CWO_ATR}), we have that

\begin{thm} \label{thm:CWO_2_leq_W_ATR_2}
$\CWO_2 \leq_W \ATR_2$.
\end{thm}
\begin{proof}
Given linear orderings $(L,M)$, define $N$ by adding a first element $0_N$ and a last element $m_N$ to $L$. Apply $\ATR_2$ to the linear ordering $N$ and the set $L \oplus M$. If we obtain an infinite descending sequence in $N$, we can use that to uniformly compute an infinite descending sequence in $L$.

Otherwise, using Proposition \ref{prop:ATR2_arith_formula}, we may assume that we obtain a hierarchy $\langle X_a\rangle_{a \in N}$ such that:
\begin{itemize}
	\item $X_{0_N} = L \oplus M$;
	\item for all $b >_N 0_N$, $X_b = \left(\bigoplus_{a <_N b} X_a\right)'''$. 
\end{itemize}

We start by attempting to use $\langle X_a\rangle_{a \in L}$-effective transfinite recursion along $L$ to define a partial $\langle X_a\rangle_{a \in L}$-recursive function $f: L \to \N$ such that $\{(a,\Phi^{X_a}_{f(a)}(0)) \in L \times M: \Phi^{X_a}_{f(a)}(0)\conv\}$ is an embedding of an initial segment of $L$ into an initial segment of $M$.

To define $f$, if we are given any $b \in L$ and $f\restriction \{a: a <_L b\}$, we need to define $f(b)$, specifically $\Phi^{X_b}_{f(b)}(0)$. First use $X_b = (\bigoplus_{a <_L b} X_a)'''$ to compute whether all of the following hold:
\begin{enumerate}
	\item for all $a <_L b$, $\Phi^{X_a}_{f(a)}(0)$ converges and outputs some element of $M$;
	\item $\{\Phi^{X_a}_{f(a)}(0): a <_L b\}$ is an initial segment of $M$;
	\item there is an $M$-least element above $\{\Phi^{X_a}_{f(a)}(0): a <_L b\}$.
\end{enumerate}
If so, we output said $M$-least element; otherwise diverge. This completes the definition of $\Phi^{X_b}_{f(b)}(0)$.

Apply the recursion theorem to the definition above to obtain a partial $\langle X_a\rangle_{a \in L}$-recursive function $f: L \to \N$. Now, to complete the definition of the backward reduction we consider the following cases.

\underline{Case 1.} $f$ is total. Then following the proof of Proposition \ref{prop:WQO_WO_ATR_CWO_ATR}, we output $\{(a,\Phi^{X_a}_{f(a)}(0)): a \in L\}$, which is an embedding from $L$ onto an initial segment of $M$.

\underline{Case 2.} There is no $L$-least element above $\{a \in L: \Phi^{X_a}_{f(a)}(0)\conv\}$. Then we can output an infinite $L$-descending sequence above $\{a \in L: \Phi^{X_a}_{f(a)}(0)\conv\}$.

\underline{Case 3.} $\{\Phi^{X_a}_{f(a)}(0): a \in L, \Phi^{X_a}_{f(a)}(0)\conv\} = M$. Then following the proof of Proposition \ref{prop:WQO_WO_ATR_CWO_ATR}, we output $\{(\Phi^{X_a}_{f(a)}(0),a): a \in L, \Phi^{X_a}_{f(a)}(0)\conv\}$, which is an embedding from $M$ onto an initial segment of $L$.

\underline{Case 4.} There is no $M$-least element above $\{\Phi^{X_a}_{f(a)}(0): a \in L, \Phi^{X_a}_{f(a)}(0)\conv~\!\}$. Then we can output an infinite $M$-descending sequence above $\{\Phi^{X_a}_{f(a)}(0): a \in L, \Phi^{X_a}_{f(a)}(0)\conv\}$.

Finally, note that the last column $X_{m_N}$ of $\langle X_a \rangle_{a \in N}$ can compute the above case division and the appropriate output for each case.
\end{proof}

\section{K\"onig's duality theorem} \label{section:KDT}

In this section, we study K\"onig's duality theorem from the point of view of computable reducibilities.

First we state some definitions from graph theory. A graph $G$ is \emph{bipartite} if its vertex set can be partitioned into two sets such that all edges in $G$ go from one of the sets to the other. It is not hard to see that $G$ is bipartite if and only if it has no odd cycle. (Hence the property of being bipartite is $\Pi^0_1$.) A \emph{matching} in a graph is a set of edges which are vertex-disjoint. A \emph{(vertex) cover} in a graph is a set of vertices which contains at least one endpoint from every edge. K\"onig's duality theorem states that:

\begin{thm}
For any bipartite graph $G$, there is a matching $M$ and a cover $C$ which are \emph{dual}, i.e., $C$ is obtained by choosing exactly one vertex from each edge in $M$. Such a pair $(C,M)$ is said to be a \emph{K\"onig cover}.
\end{thm}

K\"onig proved the above theorem for finite graphs, where it is commonly stated as ``the maximum size of a matching is equal to the minimum size of a cover''. For infinite graphs, this latter form would have little value. Instead of merely asserting the existence of a bijection, we want such a bijection to respect the structure of the graph.  Hence the notion of a K\"onig cover. Podewski and Steffens \cite{ps76} proved K\"onig's duality theorem for countable graphs. Finally, Aharoni \cite{a84} proved it for graphs of arbitrary cardinality. In this paper, we will study the theorem for countable graphs.

\begin{defn}
$\KDT$ is the following problem: given a (countable) bipartite graph $G$, produce a K\"onig cover $(C,M)$.
\end{defn}

Aharoni, Magidor, Shore \cite{ams92} studied K\"onig's duality theorem for countable graphs from the point of view of reverse mathematics. They showed that $\ATR_0$ is provable from K\"onig's duality theorem. They also showed that K\"onig's duality theorem is provable in the system $\Pi^1_1$-$\CA_0$, which is strictly stronger than $\ATR_0$. Simpson \cite{sim94} then closed the gap by showing that K\"onig's duality theorem is provable in (hence equivalent to) $\ATR_0$.

The proof of $\ATR_0$ from K\"onig's duality theorem in \cite{ams92} easily translates into a Weihrauch reduction from $\ATR$ to $\KDT$. We adapt their proof to show that $\ATR_2$ is Weihrauch reducible to $\LPO \ast \KDT$ (Theorem \ref{thm:ATR_2_leq_c_KDT}). Next, we adapt \cite{sim94}'s proof of K\"onig's duality theorem from $\ATR_0$ to show that $\KDT$ is arithmetically Weihrauch reducible to $\ATR_2$ (Theorem \ref{thm:KDT_arith_W_ATR_2}). It follows that $\ATR_2$ and $\KDT$ are arithmetically Weihrauch equivalent. Since both $\ATR_2$ and $\KDT$ have computational difficulty far above the arithmetic (see, for example, Corollary \ref{cor:WCWO_CWO_ATR2_WQO_LO_rec_instance_no_hyp_solution}), this shows that $\ATR_2$ and $\KDT$ have roughly the same computational difficulty.

Before constructing the above reductions, we make some easy observations about $\KDT$.

\begin{prop}
$\KDT \leq_W \C_{\N^\N}$, but $\C_{\N^\N}$ is not even arithmetically Weihrauch reducible to $\KDT$.
\end{prop}
\begin{proof}
The first statement holds because $\KDT$ is defined by an arithmetic predicate on an arithmetic domain. The second statement follows from Proposition \ref{prop:domain_not_Sigma11_no_arith_reduction}.
\end{proof}

\begin{prop} \label{prop:KDT_parallelizable}
$\KDT$ is parallelizable, i.e., $\widehat{\KDT} \leq_W \KDT$.
\end{prop}
\begin{proof}
This holds because the disjoint union of bipartite graphs is bipartite, and any K\"onig cover of a disjoint union of graphs restricts to a K\"onig cover on each graph.
\end{proof}

We do not know if $\ATR_2$ is parallelizable; a negative answer would separate $\ATR_2$ and $\KDT$ up to Weihrauch reducibility.

Since being a bipartite graph is a $\Pi^0_1$ property (in particular $\Pi^1_1$), we could define \emph{two-sided $\KDT$} ($\KDT_2$): given a graph, produce an odd cycle (witnessing that the given graph is not bipartite) or a K\"onig cover. This produces a problem which is Weihrauch equivalent to $\KDT$, however:

\begin{prop}
$\KDT_2 \leq_W \LPO \times \KDT$, hence $\KDT \equiv_W \KDT_2$.
\end{prop}
\begin{proof}
Given a $\KDT_2$-instance $G$ (i.e., a graph), we can uniformly compute a graph $H$ which is always bipartite and is equal to $G$ if $G$ is bipartite: $H$ has the same vertices as $G$, but as we enumerate edges of $G$ into $H$, we omit any edges that would result in an odd cycle in the graph we have enumerated thus far.

For the reduction, we apply $\LPO \times \KDT$ to $(G,H)$. If $\LPO$ (Definition \ref{defn:LPO_and_choice_problems}) tells us that $G$ is bipartite, we output a $\KDT$-solution to $H = G$. Otherwise, we can uniformly compute and output an odd cycle in $G$.

Finally, to conclude that $\KDT \equiv_W \KDT_2$, we use Proposition \ref{prop:KDT_parallelizable} and the fact that $\LPO \leq_W \KDT$, which trivially follows from Theorem \ref{thm:ATR_leq_W_KDT} later.
\end{proof}

\subsection{Reducing $\ATR_2$ to $\KDT$}

For both of our forward reductions (from $\ATR$ or $\ATR_2$ to $\KDT$), the bipartite graphs we construct are sequences of subtrees of $\N^{<\N}$. In subsection \ref{subsection:repns}, we defined these objects and described how we represent them. In this section, we will use ``tree'' as a shorthand for ``rooted subtree of $\N^{<\N}$''.

Before we describe the forward reductions in more detail, we describe our backward reduction for $\ATR \leq_W \KDT$. It only uses the cover in a K\"onig cover and not the matching. First we define a coding mechanism:

\begin{defn}
Given a tree $T$ (with root $r$) and a K\"onig cover $(C,M)$ of $T$, we can decode the bit $b$, which is the Boolean value of $r \in C$. We say that $(C,M)$ \emph{codes} $b$.

More generally, given any sequence of trees $\langle T_n: n \in X \rangle$ (with roots $r_n$) and a K\"onig cover $(C_n,M_n)$ for each $T_n$, we can uniformly decode the following set from the set $\langle (C_n,M_n) \rangle$:
\[ A = \{n \in X: r_n \in C_n\}. \]
We say that $\langle (C_n,M_n) \rangle$ \emph{codes} $A$.
\end{defn}

A priori, different K\"onig covers of the same tree or sequence of trees can code different bits or sets respectively. A tree or sequence of trees is \emph{good} if that cannot happen:

\begin{defn} \label{defn:tree_good}
A tree $T$ is \emph{good} if its root $r$ lies in $C$ for every K\"onig cover $(C,M)$ of $T$, or lies outside $C$ for every K\"onig cover $(C,M)$ of $T$. A sequence of trees $\langle T_n \rangle$ is \emph{good} if every $T_n$ is good. In other words, $\langle T_n \rangle$ is good if all of its K\"onig covers code the same set.

If $\langle T_n \rangle$ is good and every (equivalently, some) K\"onig cover of $\langle T_n \rangle$ codes $A$, we say that \emph{$\langle T_n \rangle$ codes $A$}. 

\end{defn}

We will use this coding mechanism to define the backward reduction in $\ATR \leq_W \KDT$. Here we make a trivial but important observation: for any $s \in \N^{<\N}$ and any tree $T$, the K\"onig covers of $T$ and the K\"onig covers of $s\concat T$ are in obvious correspondence, which respects whichever bit is coded. Hence $T$ is good if and only if $s\concat T$ is good.

Next, we set up the machinery for our forward reductions. Aharoni, Magidor, and Shore's \cite{ams92} proof of $\ATR_0$ from $\KDT$ uses effective transfinite recursion along the given well-ordering to construct good trees which code complicated sets. The base case is as follows:

\begin{lem} \label{lem:seq_tree_base_case}
Given any $A \subseteq \N$, we can uniformly compute a sequence of trees $\langle T_n \rangle$ which codes $A$.
\end{lem}
\begin{proof}
The tree $\{\langle \rangle\}$ codes the bit $0$. This is because any matching must be empty, hence any dual cover must be empty.

The tree $\{\langle \rangle, \langle 0 \rangle, \langle 1 \rangle\}$ codes the bit $1$. This is because any matching must contain exactly one of the two edges. Hence any cover dual to that must consist of a single node. But the root node is the only node which would cover both edges.

By defining each $T_n$ to be either of the above trees as appropriate, we obtain a sequence $\langle T_n \rangle$ which codes $A$.
\end{proof}

We may use this as the base case for our construction as well. As for the successor case, however, we want to extract extra information from the construction in \cite{ams92}. The issue is that when reducing $\ATR_2$ to $\KDT$, ``effective transfinite recursion'' on ill-founded linear orderings may produce garbage. (Of particular concern is that the resulting trees may not be good.) Nevertheless, we may attempt it anyway. If we detect inconsistencies in the resulting trees and K\"onig covers (using the extra information we have extracted), then we may use them to compute an infinite descending sequence in the given linear ordering. Otherwise, we may decode the resulting K\"onig covers	 to produce a jump hierarchy.

In order to describe our construction in detail, we need to examine the construction in \cite{ams92} closely. First we state a sufficient condition on a K\"onig cover of a tree and a node in said tree which ensures that the given K\"onig cover, when restricted to the subtree above the given node, remains a K\"onig cover. The set of all nodes satisfying the former condition form a subtree:

\begin{defn} \label{defn:T_ast}
For any tree $T$ (with root $r$) and any K\"onig cover $(C,M)$ of $T$, define the subtree $T^\ast$ (with root $r$):
\[ T^\ast = \{t \in T: \forall s(r \prec s \preceq t \rightarrow (s \notin C \lor (s\restriction (|s|-1),s) \notin M))\}. \]
\end{defn}

The motivation behind the definition of $T^\ast$ is as follows. Suppose $(C,M)$ is a K\"onig cover of $T$. If $s \in C$ and $(s\restriction (|s|-1),s) \in M$, then $C$ restricted to the subtree of $T$ above $s$ would contain $s$, but $M$ restricted to said subtree would not contain any edge with endpoint $s$. This means that the restriction of $(C,M)$ to said subtree is not a K\"onig cover. Hence we define $T^\ast$ to avoid this situation. According to \cite[Lemma 4.5]{ams92}, this is the only situation we need to avoid.

When we use the notation $T^\ast$, the cover $(C,M)$ will always be clear from context. Observe that $T^\ast$ is uniformly computable from $T$ and $(C,M)$.

\begin{lem} \label{lem:konig_cover_subtree}
For any $T$ and any K\"onig cover $(C,M)$ of $T$, define $T^\ast$ as above. Then for any $t \in T^\ast$, $(C,M)$ restricts to a K\"onig cover of the subtree of $T$ (not $T^\ast$!) above $t$.
\end{lem}
\begin{proof}
Proceed by induction on the level of $t$ using \cite[Lemma 4.5]{ams92}.
\end{proof}

Using Definition \ref{defn:T_ast} and Lemma \ref{lem:konig_cover_subtree}, we may easily show that:

\begin{prop} \label{prop:t_concat_s_in_T_ast}
Let $(C,M)$ be a K\"onig cover of $T$. Suppose that $t \in T^\ast$. Let $S$ denote the subtree of $T$ above $t$. Then $S^\ast$ is contained in $T^\ast$, where $S^\ast$ is calculated using the restriction of $(C,M)$ to $S$.
\end{prop}

Next, we define a computable operation on trees which forms the basis of the proofs of \cite[Lemmas 4.9, 4.10]{ams92}.

\begin{defn}
Given a (possibly finite) sequence of trees $\langle T_i \rangle$, each with the empty node as root, we may \emph{combine} it to form a single tree $S$, by adjoining two copies of each $T_i$ to a root node $r$. Formally,
\[ S = \{r\} \cup \{r\concat (i,0) \concat\sigma: \sigma \in T_i\} \cup \{(i,1) \concat\sigma: \sigma \in T_i\}. \]
\end{defn}

Logically, the combine operation can be thought of as $\neg\forall$:

\begin{lem} \label{lem:combine}
Suppose $\langle T_i: i \in X \rangle$ combine to form $S$. Let $r$ denote the root of $S$, and for each $i \in X$, let $r_{i,0}$ and $r_{i,1}$ denote the roots of the two copies of $T_i$ in $S$ (i.e., $r_{i,0} = r\concat (i,0)$ and $r_{i,1} = r\concat (i,1)$). Given any K\"onig cover $(C,M)$ of $S$, for each $i \in X$, we can uniformly computably choose one of $r_{i,0}$ or $r_{i,1}$ (call our choice $r_i$) such that:
\begin{itemize}
	\item $r_i \in S^\ast$;
	\item $r \notin C$ if and only if for all $i \in X$, $r_i \in C$.
\end{itemize}
Therefore if $\langle T_n: n \in X \rangle$ codes the set $A \subseteq X$, then $S$ codes the bit $0$ if and only if $A = X$.
\end{lem}
\begin{proof}
Given a K\"onig cover $(C,M)$ of $S$ and some $i \in X$, we choose $r_i$ as follows. If neither $(r,r_{i,0})$ nor $(r,r_{i,1})$ lie in $M$, then define $r_i = r_{i,0} \in S^\ast$.

Otherwise, since $M$ is a matching, exactly one of $(r,r_{i,0})$ and $(r,r_{i,1})$ lie in $M$, say $(r,r_{i,j})$. If $r \notin C$, we choose $r_i = r_{i,1-j} \in S^\ast$. If $r \in C$, note that since $(r,r_{i,j}) \in M$, we have (by duality) that $r_{i,j} \notin C$. Then we choose $r_i = r_{i,j} \in S^\ast$. This completes the definition of $r_i$.

If $r \notin C$, then for all $i \in X$ and $j<2$, $r_{i,j} \in C$ because $(r,r_{i,j})$ must be covered by $C$. In particular, $r_i \in C$ for all $i \in X$.

If $r \in C$, then (by duality) there is a unique $i \in X$ and $j<2$ such that $(r,r_{i,j}) \in M$. In that case, we chose $r_i = r_{i,j} \notin C$.
\end{proof}

In the above lemma, it is important to note that our choice of each $r_i$ depends on the K\"onig cover $(C,M)$; in fact it depends on both $C$ and $M$.

We can now use the combine operation to implement $\neg$.

\begin{defn}
The \emph{complement} of $T$, denoted $\overline{T}$, is defined by combining the single-element sequence $\langle T \rangle$.
\end{defn}

By Lemma \ref{lem:combine}, if $T$ codes the bit $i$, then $\overline{T}$ codes the bit $1-i$.

Next, we work towards iterating the combine operation to implement the jump, with the eventual goal of proving a generalization of  \cite[Lemma 4.7]{ams92}. In order to reason about trees which are formed by iterating the combine operation, we generalize Lemma \ref{lem:combine} slightly:

\begin{lem} \label{lem:uniformly_pick_r_i}
Suppose $\langle T_i: i \in X \rangle$ combine to form the subtree of $S$ above some $r \in S$. For each $i \in X$, let $r_{i,0}$ and $r_{i,1}$ denote the roots of the two copies of $T_i$ in $S$ above $r$. Given any K\"onig cover $(C,M)$ of $S$ such that $r \in S^\ast$, for each $i$, we can uniformly computably choose one of $r_{i,0}$ or $r_{i,1}$ (call our choice $r_i$) such that
\begin{itemize}
	\item $r_i \in S^\ast$;
	\item $r \notin C$ if and only if for all $i \in X$, $r_i \in C$.
\end{itemize}
\end{lem}
\begin{proof}
By Lemma \ref{lem:konig_cover_subtree}, $(C,M)$ restricts to a K\"onig cover of the subtree of $S$ above $r$. Apply Lemma \ref{lem:combine} to the subtree of $S$ above $r$, then use Proposition \ref{prop:t_concat_s_in_T_ast}.
\end{proof}

We may now present a more general and more informative version of \cite[Lemma 4.7]{ams92}.

\begin{lem} \label{lem:seq_tree_jump}
Given a sequence of trees $\langle T_i: i \in \N \rangle$ (each with the empty node as root), we can uniformly compute a sequence of trees $\langle S_e: e \in \N \rangle$ (each with the empty node as root) such that given a K\"onig cover $(C_e,M_e)$ of $S_e$, we can uniformly compute a sequence of sets of nodes $\langle R_{e,i} \rangle_i$ in $S^\ast_e$ such that
\begin{enumerate}
	\item each $r \in R_{e,i}$ has length two or three;
	\item for each $i$ and each $r \in R_{e,i}$, the subtree of $S_e$ above $r$ is $r\concat T_i$;
	\item if the set $A \subseteq \N$ is such that
\begin{align*}
i \in A \quad &\Rightarrow \quad R_{e,i} \subseteq C_e \\
i \notin A \quad &\Rightarrow \quad R_{e,i} \subseteq \overline{C_e},
\end{align*}
then $e \in A'$ if and only if the root of $S_e$ lies in $C_e$.
\end{enumerate}
Therefore, if $\langle T_i\rangle$ codes a set $A$, then $\langle S_e \rangle$ codes $A'$.
\end{lem}

Iterating the combine operation (as we will do in the following proof) introduces a complication, which necessitates the assumption in (3). For each $e$ and $i$, instead of choosing a single node $r_i$ as in Lemma \ref{lem:uniformly_pick_r_i}, we now have to choose a set of nodes $R_{e,i}$. This is because we might want to copy the tree $T_i$ more than twice, at multiple levels of the tree $S_e$. If $T_i$ is not good (Definition \ref{defn:tree_good}), these copies could code different bits (according to appropriate restrictions of $(C_e,M_e)$), so we could have $R_{e,i} \not\subseteq C_e$ and $R_{e,i} \not\subseteq \overline{C_e}$. In that case, we have little control over whether the root of $S_e$ lies in $C_e$.

Also, in the assumption of (3), we write $\Rightarrow$ instead of $\Leftrightarrow$ because writing $\Leftrightarrow$ would require us to specify separately that we do not restrict whether $i \in A$ in the case that $R_{e,i}$ is empty. (In the following proof, $R_{e,i}$ could be empty if the construction of $S_e$ does not involve $T_i$ at all.)

\begin{proof}[Proof of Lemma \ref{lem:seq_tree_jump}]
We start by constructing $S_e$. Observe that $e \in A'$ if and only if
\begin{align*}
\neg\forall (\sigma,s) \in \{(\sigma,s): \Phi^\sigma_{e,s}(e)\conv\} \neg\forall i \in \dom(\sigma)&[(\sigma(i) = 1 \land i \in A) \\
&\lor (\sigma(i) = 0 \land \neg(i \in A))].
\end{align*}
Each occurrence of $\neg\forall$ or $\neg$ corresponds to one application of the combine operation in our construction of $S_e$.

Formally, for each finite partial $\sigma: \N \to 2$ and $i \in \dom(\sigma)$, define $T^\sigma_i = T_i$ if $\sigma(i) = 1$, otherwise define $T^\sigma_i = \overline{T_i}$. Now, for each $\sigma$ and $s$ such that $\Phi^\sigma_{e,s}(e)\conv$, define $T_{\sigma,s}$ by combining $\langle T^\sigma_i: i \in \dom(\sigma)\rangle$. Finally, combine $\langle T_{\sigma,s}: \Phi^\sigma_{e,s}(e)\conv\rangle$ to form $S_e$.

Next, given a K\"onig cover $(C_e,M_e)$ of $S_e$, we construct $\langle R_{e,i} \rangle_i$ as follows. First apply Lemma \ref{lem:uniformly_pick_r_i} to $\langle T_{\sigma,s}: \Phi^\sigma_{e,s}(e)\conv\rangle$ and $(C_e,M_e)$ to choose $\langle r_{\sigma,s}: \Phi^\sigma_{e,s}(e)\conv\rangle \subseteq S^\ast_e$ such that
\begin{itemize}
	\item the subtree of $S_e$ above each $r_{\sigma,s}$ is $r_{\sigma,s}\concat T_{\sigma,s}$;
	\item the root of $S_e$ lies in $C_e$ if and only if there is some $\sigma$ and $s$ such that $\Phi^\sigma_{e,s}(e)\conv$ and $r_{\sigma,s} \notin C_e$.
\end{itemize}

Next, for each $\sigma$ and $s$ such that $\Phi^\sigma_{e,s}(e)\conv$, apply Lemma \ref{lem:uniformly_pick_r_i} to $\langle T^\sigma_i: i \in \dom(\sigma)\rangle$ and the K\"onig cover $(C_e,M_e)$ restricted to the subtree of $S_e$ above $r_{\sigma,s}$. This produces $\langle r^{\sigma,s}_i: i \in \dom(\sigma)\rangle \subseteq S^\ast_e$ (all extending $r_{\sigma,s}$) such that
\begin{itemize}
	\item the subtree of $S_e$ above each $r^{\sigma,s}_i$ is $r^{\sigma,s}_i\concat T^\sigma_i$;
	\item $r_{\sigma,s} \notin C_e$ if and only if $r^{\sigma,s}_i \in C_e$ for all $i \in \dom(\sigma)$.
\end{itemize}
Finally, for each $\sigma$ and $s$ such that $\Phi^\sigma_{e,s}(e)\conv$ and each $i$ such that $\sigma(i) = 0$, apply Lemma \ref{lem:uniformly_pick_r_i} to the single-element sequence $\langle T_i \rangle$ and $(C_e,S_e)$ restricted to the subtree of $S_e$ above $r^{\sigma,s}_i$ to obtain $\overline{r}^{\sigma,s}_i \in S^\ast_e$ extending $r^{\sigma,s}_i$ such that
\begin{itemize}
	\item the subtree of $S_e$ above $\overline{r}^{\sigma,s}_i$ is $\overline{r}^{\sigma,s}_i\concat T_i$;
	\item $r^{\sigma,s}_i \in C_e$ if and only if $\overline{r}^{\sigma,s}_i \notin C_e$.
\end{itemize}
Define
\[ R_{e,i} = \{r^{\sigma,s}_i: \Phi^\sigma_{e,s}(e)\conv,\sigma(i)=1\} \cup \{\overline{r}^{\sigma,s}_i: \Phi^\sigma_{e,s}(e)\conv,\sigma(i)=0\}. \]
First observe that each $r^{\sigma,s}_i$ has length two and each $\overline{r}^{\sigma,s}_i$ has length three. Hence (1) holds. Next, since $T^\sigma_i = T_i$ if $\sigma(i) = 1$, the subtree of $S_e$ above each $r \in R_{e,i}$ is $r\concat T_i$, i.e., (2) holds.

We prove that (3) holds. Suppose that $A \subseteq \N$ is such that
\begin{align*}
i \in A \quad &\Rightarrow \quad R_{e,i} \subseteq C_e \\
i \notin A \quad &\Rightarrow \quad R_{e,i} \subseteq \overline{C_e}.
\end{align*}
Now, $e \in A'$ if and only if there is some $\sigma \prec A$ and $s$ such that $\Phi^\sigma_{e,s}(e)\conv$. By our assumption on $A$ and the definition of $R_{e,i}$, that holds if and only if there is some $\sigma$ and $s$ such that $\Phi^\sigma_{e,s}(e)\conv$ and for all $i \in \dom(\sigma)$:
\begin{align*}
\sigma(i) = 1 \quad &\Leftrightarrow \quad r^{\sigma,s}_i \in C_e \\
\sigma(i) = 0\quad &\Leftrightarrow \quad \overline{r}^{\sigma,s}_i \notin C_e .
\end{align*}
Chasing through the above definitions, we see that the above holds if and only if the root of $S_e$ lies in $C_e$, as desired.

Finally, suppose that $\langle T_i \rangle$ codes the set $A$. We show that $\langle S_e \rangle$ codes $A'$. Fix a K\"onig cover $\langle (C_e,M_e) \rangle$ of $\langle S_e \rangle$. First we show that the assumption in (3) holds for $A$. Fix $e,i \in \N$. If $R_{e,i}$ is empty, the desired statement holds. Otherwise, fix $r \in R_{e,i}$. Since $r$ lies in $S^\ast_e$, Lemma \ref{lem:konig_cover_subtree} says that $(C_e,M_e)$ restricts to a K\"onig cover of the subtree of $S_e$ above $r$. By (2), the subtree of $S_e$ above $r$ is $r\concat T_i$. Since $T_i$ codes $A(i)$, so does $r\concat T_i$. We conclude that
\[ r \in C_e \quad \Leftrightarrow \quad \text{ the root of }T_i \in C_i \quad \Leftrightarrow \quad i \in A. \]
It follows that the assumption in (3) holds for $A$. Now by (3), $e \in A'$ if and only if the root of $S_e$ lies in $C_e$.

Since this holds for every K\"onig cover $\langle (C_e,M_e) \rangle$ of $\langle S_e \rangle$, $\langle S_e \rangle$ codes $A'$ as desired.
\end{proof}

\begin{rmk}
In the proof of Lemma \ref{lem:seq_tree_jump}, we could just as well have defined $R_{e,i}$ to be the set of all nodes in $S^\ast_e$ which are roots of copies of $T_i$. (Formally, for each $T_{\sigma,s}$ such that $\Phi^\sigma_{e,s}(e)\conv$, we could include the roots of the component $T^\sigma_i$'s if $\sigma(i) = 1$, and the roots of the component $T_i$'s in the $T^\sigma_i$'s if $\sigma(i) = 0$, as long as they lie in $S^\ast_e$.)	
\end{rmk}

Next, we make two small tweaks to Lemma \ref{lem:seq_tree_jump}. First, we adjust conclusion (3) to fit our definition of jump hierarchy (Definition \ref{defn:jump_hierarchy}). Second, we broaden the scope of our conclusions to include K\"onig covers of copies of $S_n$, not just K\"onig covers of $S_n$ itself. Lemma \ref{lem:seq_tree_join_jump} is the central lemma behind our reductions from $\ATR$ and $\ATR_2$ to $\KDT$.

\begin{lem} \label{lem:seq_tree_join_jump}
Given a sequence of sequences of trees $\langle \langle T^a_n \rangle_n\rangle_a$ (each with the empty node as root), we can uniformly compute a sequence of trees $\langle S_n \rangle_n$ (each with the empty node as root) such that for any $s_n \in \N^{<\N}$ and any K\"onig cover $(C_n,M_n)$ of $s_n\concat S_n$, we can uniformly compute a sequence of sets of nodes $\langle R^a_{n,i} \rangle_{a,i}$ in $(s_n\concat S_n)^\ast$ such that
\begin{enumerate}
	\item each $r \in R^a_{n,i}$ has length two or three (plus the length of $s_n$);
	\item for each $a$, $i$, and each $r \in R^a_{n,i}$, the subtree of $s_n\concat S_n$ above $r$ is $r\concat T^a_i$;
	\item
suppose that for each $a$, the set $Y_a \subseteq \N$ is such that
\begin{align*}
i \in Y_a \quad &\Rightarrow \quad R^a_{n,i} \subseteq C_n \\
i \notin Y_a \quad &\Rightarrow \quad R^a_{n,i} \subseteq \overline{C_n},
\end{align*}
then $n \in \left(\bigoplus_a Y_a\right)'$ if and only if $s_n$ lies in $C_n$.
\end{enumerate}
Therefore, if for each $a$, $\langle T^a_n \rangle_n$ codes a set $Y_a$, then $\langle S_n \rangle_n$ codes $\left(\bigoplus_a Y_a\right)'$.
\end{lem}
\begin{proof}
Apply Lemma \ref{lem:seq_tree_jump} to $\langle T^a_n \rangle_{a,n}$. Given a K\"onig cover $(C_n,M_n)$ of $s_n\concat S_n$, we may compute the corresponding K\"onig cover of $S_n$ (as we observed after Definition \ref{defn:tree_good}). Then apply Lemma \ref{lem:seq_tree_jump} to obtain $\langle R^a_{n,i} \rangle_{n,i}$ in $S_n^\ast$. It is straightforward to check that $\langle s_n\concat R^a_{n,i} \rangle_{n,i}$ satisfies conclusions (1)--(3).
\end{proof}

As a warmup for our reduction from $\ATR_2$ to $\KDT$, we use Lemma \ref{lem:seq_tree_join_jump} to prove that $\ATR \leq_W \KDT$. Our proof is essentially the same as that of \cite[Theorem 4.11]{ams92}. Note that we do not use the sets $R^a_{n,i}$ in the following proof, only the final conclusion of Lemma \ref{lem:seq_tree_join_jump}. (The sets $R^a_{n,i}$ will be used in our reduction from $\ATR_2$ to $\KDT$.)

\begin{thm} \label{thm:ATR_leq_W_KDT}
$\ATR \leq_W \KDT$.
\end{thm}
\begin{proof}
We reduce the version of $\ATR$ in Proposition \ref{prop:ATR_equiv_ATR_labels} to $\KDT$. Given a labeled well-ordering $\L$ and a set $A$, we will use $(\L \oplus A)$-effective transfinite recursion on $L$ to define an $(\L \oplus A)$-recursive function $f: L \to \omega$ such that for each $b \in L$, $\Phi^{\L \oplus A}_{f(b)}$ is interpreted as a sequence of trees $\langle T^b_n \rangle_n$ (each with the empty node as root). We will show that $\langle T^b_n \rangle_n$ codes the $b^{\text{th}}$ column of the jump hierarchy on $L$ which starts with $A$.

For the base case, we use Lemma \ref{lem:seq_tree_base_case} to compute a sequence of trees $\langle T^{0_L}_n \rangle_n$ which codes $A$. Otherwise, for $b >_L 0_L$, we use Lemma \ref{lem:seq_tree_join_jump} to compute a sequence of trees $\langle T^b_n \rangle_n$ such that if for each $a <_L b$, $\Phi^{\L \oplus A}_{f(a)}$ is (interpreted as) a sequence of trees $\langle T^a_n \rangle_n$ which codes $Y_a$, then $\langle T^b_n \rangle_n$ codes $\left(\bigoplus_{a <_L b} Y_a\right)'$.

Note that $f$ is total: for any $b$, we can interpret $\langle \Phi^{\L \oplus A}_{f(a)} \rangle_{a <_L b}$ as a sequence of sequences of trees and apply Lemma \ref{lem:seq_tree_join_jump} to obtain $\langle T^b_n \rangle_n$. This also means that every $\langle T^b_n \rangle_n$ (for $b >_L 0_L$) was obtained using Lemma \ref{lem:seq_tree_join_jump}.

We may view the disjoint union of $\langle \langle T^b_n \rangle_n \rangle_{b \in L}$ as a $\KDT$-instance. This defines the forward reduction from $\ATR$ to $\KDT$.

For the backward reduction, let $\langle \langle (C^b_n,M^b_n) \rangle_n \rangle_{b \in L}$ be a solution to the above $\KDT$-instance. We may uniformly decode said solution to obtain a sequence of sets $\langle Y_b \rangle_{b \in L}$.

By transfinite induction along $L$ using Lemmas \ref{lem:seq_tree_base_case} and \ref{lem:seq_tree_join_jump}, $\langle T^b_n \rangle_n$ is good for all $b \in L$, and $\langle Y_b \rangle_{b \in L}$ is the 	jump hierarchy on $L$ which starts with $A$.
\end{proof}

What if we want to use the forward reduction from $\ATR$ to $\KDT$ in our reduction from $\ATR_2$ to $\KDT$? If the given $\ATR_2$-instance $\L$ is ill-founded, things could go wrong in the ``effective transfinite recursion''. Specifically, there may be some $a \in L$ and $i \in \N$ such that $T^a_i$ is not good, i.e., there may be some $r,s \in \N^{<\N}$ and some K\"onig covers of $r\concat T^a_i$ and $s\concat T^a_i$ which code different bits. In order to salvage the situation, we will modify the backward reduction to check for such inconsistencies. If they are present, we use them to compute an infinite $<_L$-descending sequence.

In order to detect inconsistencies, for each $b \in L$ and $n \in \N$, we need to keep track of the internal structure of $(C^b_n,M^b_n)$ in the $\KDT$-solution. According to Lemma \ref{lem:seq_tree_join_jump} and our construction of $T^b_n$, for each $a <_L b$ and $i \in \N$, there is a set of nodes $R^a_{n,i}$ in $(T^b_n)^\ast$ such that:
\begin{itemize}
	\item for each $r \in R^a_{n,i}$, the subtree of $T^b_n$ above $r$ is $r\concat T^a_i$;
	\item if for each $i$, either $R^a_{n,i} \subseteq C^b_n$ or $R^a_{n,i} \subseteq \overline{C^b_n}$, then $(C^b_n,M^b_n)$ codes the $n^{\text{th}}$ bit of $(\bigoplus_a Y_a)'$, where each $Y_a$ satisfies the assumption in Lemma \ref{lem:seq_tree_join_jump}(3).
\end{itemize}

The ``consistent'' case is if for each $a <_L b$ and $i \in \N$, $(C^a_i,T^a_i)$ codes the same bit as the restriction of $(C^b_n,M^b_n)$ to the subtree above each $r$ in $R^a_{n,i}$. (This must happen if each $T^a_i$ is good, but it could also happen ``by chance''.) We will show that this ensures that for each $a$ and $i$, either $R^a_{n,i} \subseteq C^b_n$ or $R^a_{n,i} \subseteq \overline{C^b_n}$. Furthermore, for each $a$, the $Y_a$ coded by $\langle T^a_i \rangle_i$ must satisfy the assumptions in Lemma \ref{lem:seq_tree_join_jump}(3), so we correctly calculate the next column of our jump hierarchy.

On the other hand, what if there are some $a <_L b$, $i \in \N$, and $r_0 \in R^a_{n,i}$ such that $(C^a_i,M^a_i)$ codes a different bit from the restriction of $(C^b_n,M^b_n)$ to the subtree above $r_0$? Then consider $T^a_i$ and the subtree of $T^b_n$ above $r_0$. The latter tree is a copy of $T^a_i$ (specifically, it is $r_0\concat T^a_i$), yet its K\"onig cover codes a different bit from that of $T^a_i$, so we can use Lemma \ref{lem:seq_tree_join_jump} to find a subtree of $T^a_i$ and a subtree of $T^b_n$ above $r_0$ (both subtrees are copies of $T^{a_0}_{i_0}$ for some $a_0 <_L a$, $i_0 \in \N$) on which appropriate restrictions of $(C^a_i,M^a_i)$ and $(C^b_n,M^b_n)$ code different bits. By repeating this process, we can obtain an infinite $<_L$-descending sequence.

In order to formalize the above arguments, we organize the above recursive process using the sets $R^{b,a}_{n,i}$, defined as follows:

\begin{defn} \label{defn:R_ba_ni}
Fix a labeled linear ordering $\L$ and use the forward reduction in Theorem \ref{thm:ATR_leq_W_KDT} to compute $\langle \langle T^b_n \rangle_n \rangle_{b \in L}$. For each $n$ and $b$, fix a K\"onig cover $(C^b_n,M^b_n)$ of $T^b_n$. For each $a <_L b$ and each $i,n \in \N$, we define a set of nodes $R^{b,a}_{n,i}$ in $T^b_n$ as follows: $R^{b,a}_{n,i}$ is the set of all $r$ for which there exist $j \geq 1$ and
\begin{center}
\begin{tabular}{r c c c c c l l}
$\langle \rangle = r_0$ & $\prec$ & $r_1$ & $\prec$ &$\cdots$ & $\prec$ & $r_j = r$ & in $T^b_n$ \\
$b = c_0$ & $>_L$ & $c_1$ & $>_L$ &$\cdots$ & $>_L$ & $c_j = a$ & in $L$ \\
$n = i_0$ & , & $i_1$ & , &$\cdots$ & , & $i_j = i$ & in $\N$
\end{tabular}
\end{center}
such that for all $0 < l \leq j$, $r_l$ lies in $R^{c_l}_{i_{l-1},i_l}$ as calculated by applying Lemma \ref{lem:seq_tree_join_jump} to $(C^b_n,M^b_n)$ restricted to the subtree of $T^b_n$ above $r_{l-1}$.
\end{defn}

We make two easy observations about $R^{b,a}_{n,i}$:
\begin{enumerate}
	\item By induction on $l$, $r_l$ lies in $(T^b_n)^\ast$ and the subtree of $T^b_n$ above $r_l$ is $r_l\concat T^{c_l}_{i_l}$. In particular, for each $r \in R^{b,a}_{n,i}$, $r \in (T^b_n)^\ast$ and the subtree of $T^b_n$ above $r$ is $r\concat T^a_i$.
	\item $R^{b,a}_{n,i}$ is uniformly c.e.\ in $\L \oplus (C^b_n,M^b_n)$. (A detailed analysis shows that $R^{b,a}_{n,i}$ is uniformly computable in $\L \oplus (C^b_n,M^b_n)$, but we do not need that.)
\end{enumerate}

With the $R^{b,a}_{n,i}$'s in hand, we can make precise what we mean by consistency:

\begin{defn}
In the same context as the previous definition, we say that $a \in L$ is \emph{consistent} if for all $i \in \N$:
\begin{align*}
\text{the root of }T^a_i \in C^a_i \quad &\Rightarrow \quad R^{b,a}_{n,i} \subseteq C^b_n \text{ for all }b >_L a,n \in \N \\
\text{the root of }T^a_i \notin C^a_i \quad &	\Rightarrow \quad R^{b,a}_{n,i} \subseteq \overline{C^b_n} \text{ for all }b >_L a,n \in \N.
\end{align*}
\end{defn}

Observe that if $T^a_i$ is good for all $i$, then observation (1) above implies that $a$ is consistent, regardless of what $\langle (C^b_n,M^b_n) \rangle_{b,n}$ may be. However, unless $L$ is well-founded, we cannot be certain that $T^a_i$ is good. Consistency is a weaker condition which suffices to ensure that we can still obtain a jump hierarchy on $L$, as we show in Corollary \ref{cor:consistent_elt_cut}. We will also show that inconsistency cannot come from nowhere, i.e., if $b_0$ is inconsistent, then there is some $b_1 <_L b_0$ which is inconsistent, and so on, yielding an infinite $<_L$-descending sequence of inconsistent elements.

Furthermore, consistency is easy to check: by observation (2) above, whether $a$ is consistent is $\Pi^0_1$ (in $\L \oplus \langle (C^b_n,M^b_n)\rangle_{b,n}$).

We prove two lemmas that will yield the desired result when combined:

\begin{lem} \label{lem:Y_a_cohesive_code_Y_b}
Fix K\"onig covers $\langle (C^b_n,M^b_n) \rangle_{b,n}$ for $\langle T^b_n \rangle_{b,n}$. Now fix $n$ and $b$. Suppose that for each $a <_L b$, the set $Y_a \subseteq \N$ is such that
\begin{align*}
i \in Y_a \quad &\Rightarrow \quad R^{b,a}_{n,i} \subseteq C^b_n \\
i \notin Y_a \quad &\Rightarrow \quad R^{b,a}_{n,i} \subseteq \overline{C^b_n}.
\end{align*}
Then $n \in \left(\bigoplus_{a <_L b} Y_a\right)'$ if and only if the root of $T^b_n$ lies in $C^b_n$.
\end{lem}
\begin{proof}
Recall that $\langle T^b_n \rangle_{n \in \N}$ is computed by applying Lemma \ref{lem:seq_tree_join_jump} to $\langle \langle T^a_n \rangle_{n \in \N} \rangle_{a <_L b}$. By definition of $R^{b,a}_{n,i}$, $R^a_{n,i}$ (as obtained from Lemma \ref{lem:seq_tree_join_jump}) is a subset of $R^{b,a}_{n,i}$ (this is the case $j = 1$). So for all $a <_L b$,
\begin{align*}
i \in Y_a \quad &\Rightarrow \quad R^a_{n,i} \subseteq R^{b,a}_{n,i} \subseteq C^b_n \\
i \notin Y_a \quad &\Rightarrow \quad R^a_{n,i} \subseteq R^{b,a}_{n,i} \subseteq \overline{C^b_n}.
\end{align*}
The desired result follows from Lemma \ref{lem:seq_tree_join_jump}(3).
\end{proof}

\begin{lem} \label{lem:cohesive_induction}
Fix K\"onig covers $\langle (C^c_m,M^c_m) \rangle_{c,m}$ for $\langle T^c_m \rangle_{c,m}$. Now fix $m$ and $b <_L c$. Suppose that for each $a <_L b$, the set $Y_a \subseteq \N$ is such that
\begin{align*}
i \in Y_a \quad &\Rightarrow \quad R^{c,a}_{m,i} \subseteq C^c_m \\
i \notin Y_a \quad &\Rightarrow \quad R^{c,a}_{m,i} \subseteq \overline{C^c_m}.
\end{align*}
Then for all $n \in \N$,	
\begin{align*}
n \in \left(\bigoplus_{a <_L b} Y_a\right)' \quad &\Rightarrow \quad R^{c,b}_{m,n} \subseteq C^c_m \\
n \notin \left(\bigoplus_{a <_L b} Y_a\right)' \quad &\Rightarrow \quad R^{c,b}_{m,n} \subseteq \overline{C^c_m}.
\end{align*}
\end{lem}
\begin{proof}
If $R^{c,b}_{m,n}$ is empty, then the desired result is vacuously true. Otherwise, consider $r \in R^{c,b}_{m,n}$. As we observed right after Definition \ref{defn:R_ba_ni}, $r \in (T^c_m)^\ast$ and the subtree of $T^c_m$ above $r$ is $r\concat T^b_n$. $T^b_n$ was constructed by applying Lemma \ref{lem:seq_tree_join_jump} to $\langle \langle T^a_n \rangle_{n \in \N} \rangle_{a <_L b}$, so we can use the restriction of $(C^c_m, M^c_m)$ to $r\concat T^b_n$ to compute sets $\langle R^a_{n,i} \rangle_{a <_L b,i \in \N}$ of nodes in $(r\concat T^b_n)^\ast$ satisfying the conclusions of Lemma \ref{lem:seq_tree_join_jump}.

We claim that for all $a <_L b$, $R^a_{n,i} \subseteq R^{c,a}_{m,i}$.

\begin{proof}[Proof of claim]
Consider $s\in R^a_{n,i}$. We know that $s$ extends $r$ and $r \in R^{c,b}_{m,n}$. Fix $j \geq 1$ and
\begin{center}
\begin{tabular}{r c c c c c l l}
$\langle \rangle = r_0$ & $\prec$ & $r_1$ & $\prec$ &$\cdots$ & $\prec$ & $r_j = r$ & in $T^c_m$ \\
$c = c_0$ & $>_L$ & $c_1$ & $>_L$ &$\cdots$ & $>_L$ & $c_j = b$ & in $L$ \\
$m = i_0$ & , & $i_1$ & , &$\cdots$ & , & $i_j = n$ & in $\N$
\end{tabular}
\end{center}
which witness that $r \in R^{c,b}_{m,n}$. Then we can append one column:
\begin{center}
\begin{tabular}{r c c c c c l l l l}
$\langle \rangle = r_0$ & $\prec$ & $r_1$ & $\prec$ &$\cdots$ & $\prec$ & $r_j = r$ & $\prec$ & $r_{j+1} = s$ & in $T^c_m$ \\
$c = c_0$ & $>_L$ & $c_1$ & $>_L$ &$\cdots$ & $>_L$ & $c_j = b$ & $>_L$ & $c_{j+1} = a$ & in $L$ \\
$m = i_0$ & , & $i_1$ & , &$\cdots$ & , & $i_j = n$ & , & $i_{j+1} = i$ & in $\N$
\end{tabular}
\end{center}
Since $s \in R^a_{n,i}$, this witnesses that $s \in R^{c,a}_{m,i}$.
\end{proof}

By our claim, we have that
\begin{align*}
i \in Y_a \quad &\Rightarrow \quad R^a_{n,i} \subseteq R^{c,a}_{m,i} \subseteq C^c_m \\
i \notin Y_a \quad &\Rightarrow \quad R^a_{n,i} \subseteq R^{c,a}_{m,i} \subseteq \overline{C^c_m}.
\end{align*}
By Lemma \ref{lem:seq_tree_join_jump}(3), $n \in \left(\bigoplus_{a <_L b} Y_a\right)'$ if and only if $r \in C^c_m$. This concludes the proof.
\end{proof}

Putting the previous two lemmas together, we obtain

\begin{cor} \label{cor:consistent_elt_cut}
Fix K\"onig covers $\langle (C^b_n,M^b_n) \rangle_{b,n}$ for $\langle T^b_n \rangle_{b,n}$. For each $b \in L$, define $Y_b$ by decoding $\langle (C^b_n,M^b_n) \rangle_n$, i.e.,
\[ Y_b = \{n \in \N: \text{the root of }T^b_n \text{ lies in }C^b_n\}. \]
If all $a <_L b$ are consistent, then $b$ is consistent and $Y_b = \left(\bigoplus_{a <_L b} Y_a\right)'$.
\end{cor}
\begin{proof}
$0_L$ is consistent because every $T^{0_L}_n$ is good (Lemma \ref{lem:seq_tree_base_case}). Consider now any $b >_L 0_L$. Every $a <_L b$ is consistent, so for all $a <_L b$:
\begin{align*}
i \in Y_a \quad &\Rightarrow \quad R^{c,a}_{m,i} \subseteq C^c_m \text{ for all }c >_L a,m \in \N \\
i \notin Y_a \quad &	\Rightarrow \quad R^{c,a}_{m,i} \subseteq \overline{C^c_m} \text{ for all }c >_L a,m \in \N.
\end{align*}
By Lemma \ref{lem:Y_a_cohesive_code_Y_b}, $Y_b = \left(\bigoplus_{a <_L b} Y_a\right)'$.

Also, by Lemma \ref{lem:cohesive_induction}, for all $n \in \N$:
\begin{align*}
n \in \left(\bigoplus_{a <_L b} Y_a\right)' \quad &\Rightarrow \quad R^{c,b}_{m,n} \subseteq C^c_m \text{ for all }c >_L b, m \in \N \\
n \notin \left(\bigoplus_{a <_L b} Y_a\right)' \quad &\Rightarrow \quad R^{c,b}_{m,n} \subseteq \overline{C^c_m} \text{ for all }c >_L b, m \in \N.
\end{align*}
It follows that $b$ is consistent.
\end{proof}

We are finally ready to construct a reduction from $\ATR_2$ to $\KDT$.

\begin{thm} \label{thm:ATR_2_leq_c_KDT}
$\ATR_2 \leq_W \LPO \ast \KDT$. In particular, $\ATR_2 \leq_c \KDT$ and $\ATR_2 \leq_W^{\arith} \KDT$.
\end{thm}
\begin{proof}
Given a labeled linear ordering $\L$ (we may assume that $L$ is labeled by Proposition \ref{prop:ATR_2_equiv_ATR_2_labels}) and a set $A$, we apply the forward reduction in Theorem \ref{thm:ATR_leq_W_KDT} to produce some $\KDT$-instance $\langle T^b_n \rangle_{b,n}$. For the backward reduction, given a $\KDT$-solution $\langle \langle (C^b_n,M^b_n) \rangle_n \rangle_{b \in L}$, we start by uniformly decoding it to obtain a sequence of sets $\langle Y_b \rangle_{b \in L}$.

Next, since $R^{b,a}_{n,i}$ is uniformly c.e.\ in $\L \oplus (C^b_n,M^b_n)$, whether some $a \in L$ is inconsistent is uniformly c.e.\ in $\L \oplus \langle (C^b_n,M^b_n) \rangle_{b,n}$. Therefore we can use $\LPO$ (Definition \ref{defn:LPO_and_choice_problems}) to determine whether every $a \in L$ is consistent.

If so, by Corollary \ref{cor:consistent_elt_cut}, $\langle Y_b \rangle_{b \in L}$ is a jump hierarchy on $L$ which starts with $A$.

If not, by Corollary \ref{cor:consistent_elt_cut}, every inconsistent element is preceded by some other inconsistent element. Since whether some $a \in L$ is inconsistent is uniformly c.e.\ in $\L \oplus \langle (C^b_n,M^b_n) \rangle_{b,n}$, we can use it to compute an infinite $<_L$-descending sequence of inconsistent elements.
\end{proof}

\subsection{Reducing $\KDT$ to $\ATR_2$}

This section presumes an understanding of the proofs in Simpson \cite{sim94}. First, he proved in $\ATR_0$ that for any set $G$, there is a countable coded $\omega$-model of $\Sigma^1_1$-$\AC$ which contains $G$. His proof \cite[Lemma 1]{sim94} also shows that

\begin{lem} \label{lem:jump_hierarchy_proper_cut_SigmaAC}
If $\langle X_a \rangle_{a \in L}$ is a jump hierarchy on $L$ and $I$ is a proper cut of $L$ which is not computable in $\langle X_a \rangle_{a \in L}$, then the countable coded $\omega$-model $\M = \{A: \exists a \in I(A \leq_T X_a)\}$ satisfies $\Sigma^1_1$-$\AC$.
\end{lem}
\begin{proof}[Sketch of proof]
Given an instance $\phi(n,Y)$ of $\Sigma^1_1$-$\AC$, for each $n$, let $a_n \in I$ be $<_L$-least such that $X_{a_n}$ computes a solution to $\phi(n,\cdot)$. Since $I$ is a proper cut, for any $a \in I$ and $b \in L\backslash I$, $X_b$ computes every $X_a$-hyperarithmetic set. Therefore if $b \in L\backslash I$, then $X_b$ computes $(a_n)_{n \in \omega}$.

Hence $(a_n)_{n \in \omega}$ is not cofinal in $I$, otherwise $I$ would be computable in $\langle X_a \rangle_{a \in L}$. Fix $b \in I$ which bounds $(a_n)_{n \in \omega}$. Then there is a $\Sigma^1_1$-$\AC$-solution to $\phi$ which is arithmetic in $X_b$ (and hence lies in $\M$), as desired.
\end{proof}

We now adapt \cite{sim94}'s proof of K\"onig's duality theorem in $\ATR_0$ to show that

\begin{thm} \label{thm:KDT_arith_W_ATR_2}
$\KDT$ is arithmetically Weihrauch reducible to $\ATR_2$.
\end{thm}
\begin{proof}
Given a bipartite graph $G$, we would like to use $\ATR_2$ to produce a countable coded $\omega$-model of $\Sigma^1_1$-$\AC$ which contains $G$. In order to do that, we define a $G$-computable linear ordering (i.e., an instance of $\ATR_2$) using the recursion theorem, as follows.

First define a predicate $P(G,e,X)$ to hold if $X$ is a jump hierarchy on $L^G_e$ which starts with $G$ and does not compute any proper cut in $L^G_e$. Notice that $P(G,e,X)$ is arithmetic.

The total $G$-computable function to which we apply the recursion theorem is as follows. Given any $G$-computable linear ordering $L^G_e$, consider the $G$-computable tree $H^G_e$ whose paths (if any) are solutions to $P(G,e,\cdot)$ (with Skolem functions). Then output an index for the Kleene-Brouwer ordering of $H^G_e$.

By the recursion theorem, we can $G$-uniformly compute a fixed point $e$ for the above computable transformation. Observe that the following are (consecutively) equivalent:
\begin{enumerate}
	\item $L^G_e$ has an infinite $G$-hyperarithmetic descending sequence;
	\item $H^G_e$ has a $G$-hyperarithmetic path;
	\item $P(G,e,\cdot)$ has a $G$-hyperarithmetic solution, i.e., there is a $G$-hyperarithmetic jump hierarchy on $L^G_e$ which starts with $G$ and does not compute any proper cut in $L^G_e$;
	\item $L^G_e$ is well-founded.
\end{enumerate}
(The only nontrivial implication is (3) $\Rightarrow$ (4), which holds because no jump hierarchy on a $G$-computable ill-founded linear ordering can be $G$-hyperarithmetic.) But (1) and (4) contradict each other, so (1)--(4) are all false. Hence $L^G_e$ must be ill-founded and cannot have any infinite $G$-hyperarithmetic descending sequence. It follows that every infinite $L^G_e$-descending sequence defines a proper cut in $L^G_e$.

Next, we show that given an $\ATR_2$-solution to $L^G_e$, we can arithmetically uniformly compute some proper cut $I$ in $L^G_e$ and a solution to $P(G,e,\cdot)$, i.e., a jump hierarchy $\langle X_a \rangle_{a \in L^G_e}$ which does not compute any proper cut in $L^G_e$. Then by Lemma \ref{lem:jump_hierarchy_proper_cut_SigmaAC}, the countable coded $\omega$-model of all sets which are computable in some $X_a$, $a \in I$, satisfies $\Sigma^1_1$-$\AC$ as desired.

If $\ATR_2$ gives us an infinite $L^G_e$-descending sequence $S$, then we can use $S$ to arithmetically uniformly compute a proper cut in $L^G_e$. Since $L^G_e$ is the Kleene-Brouwer ordering of $H^G_e$, we can also use $S$ to arithmetically uniformly compute a path on $H^G_e$. From said path, we can uniformly compute a solution to $P(G,e,\cdot)$.

If $\ATR_2$ gives us a jump hierarchy $X$ on $L^G_e$, we show how to arithmetically uniformly compute an infinite $L^G_e$-descending sequence. We may then proceed as in the previous case.

First arithmetically uniformly check whether $X$ computes any proper cut in $L^G_e$. If so, we can arithmetically uniformly find an index for such a computation, and produce a proper cut in $L^G_e$. From that, we may uniformly compute an infinite $L^G_e$-descending sequence. If not, then $X$ is a solution to $P(G,e,\cdot)$, so we can arithmetically uniformly compute a path on $H^G_e$, and hence an infinite $L^G_e$-descending sequence.

We have produced a countable coded $\omega$-model of $\Sigma^1_1$-$\AC$ which contains the given graph $G$. Call it $\M$.

With $\M$ in hand, we follow the rest of Simpson's proof in order to obtain a $\KDT$-solution to $G$. His idea is to ``relativize'' Aharoni, Magidor, Shore's \cite{ams92} proof of $\KDT$ in $\Pi^1_1$-$\CA_0$ to $\M$. In the following, we will often write $\M$ instead of ``the code of $\M$''.

Let $G = (X,Y,E)$. (If we are not given a partition $(X,Y)$ of the vertex set of $G$ witnessing that $G$ is bipartite, we can arithmetically uniformly compute such a partition.) Recall a definition from \cite{ams92}: if $A \subseteq X$, then the \emph{demand set} is defined by
\[ D_G(A) = \{y \in Y: xEy \rightarrow x \in A\}. \]
Note that if $A \in \M$, then $D_G(A)$ is uniformly arithmetic in $\M$ and the code of $A$.

Next, consider the set of pairs
\[ S = \{(A,F) \in \M: A \subseteq X \text{ and }F: A \to D_G(A) \text{ is a matching}\}. \]
(Note that $A$ and $F$ may be infinite.) $S$ (specifically the set of codes of $(A,F) \in S$) is arithmetic over $\M$. So is the set $\bigcup\{A: (A,F) \in S\} \subseteq X$, which we denote by $A^\ast$.

Next, for each $x \in A^\ast$, we define $F^\ast(x)$ to be $F(x)$, where $(A,F)$ is the least (with respect to the enumeration of $\M$) pair in $S$ such that $x \in A$. Then $F^\ast: A^\ast \to D_G(A^\ast)$ is a matching (\cite[Lemma 2]{sim94}). Note that $F^\ast$ is arithmetic over $\M$.

Next, define $X^\ast = X-A^\ast$ and $Y^\ast = Y-D_G(A^\ast)$. Both sets are arithmetic over $\M$. Simpson then constructs (by recursion along $\omega$) a matching $H$ from $Y^\ast$ to $X^\ast$ which is arithmetic in $G \oplus \M$, as follows. Each step of the recursion proceeds by searching for a pair of adjacent vertices (one in $X^\ast$, one in $Y^\ast$) whose removal does not destroy \emph{goodness}: a cofinite induced subgraph $G'$ (with vertices partitioned into $X' \subseteq X$ and $Y' \subseteq Y$) of $G$ is \emph{good} if for any $A \subseteq X'$ in $\M$ and any matching $F: A \to D_{G'}(A)$ in $\M$, $D_{G'}(A)-\range(F)$ and $Y^\ast$ are disjoint. (This definition is not related to Definition \ref{defn:tree_good}.) This recursion eventually matches every vertex in $Y^\ast$ to some vertex in $X^\ast$ (\cite[Lemmas 3, 5]{sim94}).

The property of goodness (where each $G'$ is encoded by the finite set of vertices in $G\backslash G'$) is arithmetic over $\M$. Hence the resulting matching $H$ is arithmetic over $\M$.

Finally, we arrive at a $\KDT$-solution to $G$: $F^\ast \cup H$ is a matching in $G$, with corresponding dual cover $A^\ast \cup Y^\ast$. $(F^\ast \cup H,A^\ast \cup Y^\ast)$ can be arithmetically uniformly computed from $\M$.
\end{proof}

Using Theorems \ref{thm:ATR_2_leq_c_KDT} and \ref{thm:KDT_arith_W_ATR_2}, we conclude that

\begin{cor}
$\ATR_2$ and $\KDT$ are arithmetically Weihrauch equivalent.
\end{cor}

\bibliographystyle{plain}
\bibliography{atr_weihrauch}

\end{document}